\documentclass[11pt,reqno,twoside]{article}
\usepackage{amsmath,amsfonts,amsthm,amssymb}
\usepackage{graphics}
\usepackage[all]{xy}
\usepackage{color}
\theoremstyle{definition}
\newtheorem{theorem}{Theorem}[section]

\newtheorem{corollary}[theorem]{Corollary}
\newtheorem{lemma}[theorem]{Lemma}
\newtheorem{remark}[theorem]{Remark}

\begin{document}

\title{\normalsize\bf ON GENERATING SETS OF YOSHIKAWA MOVES FOR MARKED GRAPH DIAGRAMS OF SURFACE-LINKS}

\author{\small 
JIEON KIM and YEWON JOUNG
\smallskip\\
{\small\it 
Department of Mathematics, Graduate School of Natural Sciences,
}\\ 
{\small\it 
Pusan National University, Busan 609-735, Korea
}\\
{\small\it jieonkim@pusan.ac.kr}\\
{\small\it yewon112@pusan.ac.kr}
\smallskip\\
\smallskip\\
{\small SANG YOUL LEE}
\smallskip\\
{\small\it 
Department of Mathematics, Pusan National University,
}\\ 
{\small\it Busan 609-735, Korea}\\
{\small\it sangyoul@pusan.ac.kr}}

\renewcommand\leftmark{\centerline{\footnotesize 
J. Kim, Y. Joung \& S. Y. Lee}}
\renewcommand\rightmark{\centerline{\footnotesize 
On generating sets of Yoshikawa moves for marked graph diagrams of surface-links}}

\maketitle

\begin{abstract}
A marked graph diagram is a link diagram possibly with marked $4$-valent vertices. S. J. Lomonaco, Jr. and K. Yoshikawa introduced a method of representing surface-links by marked graph diagrams. Specially, K. Yoshikawa suggested local moves on marked graph diagrams, nowadays called Yoshikawa moves. It is now known that two marked graph diagrams representing equivalent surface-links are related by a finite sequence of these Yoshikawa moves. In this paper, we provide some generating sets of Yoshikawa moves on marked graph diagrams representing unoriented surface-links, and also oriented surface-links. We also discuss independence of certain Yoshikawa moves from the other moves.
\end{abstract}





\section{Introduction}
\label{intro}

By a {\it surface-link} we mean a closed 2-manifold smoothly (or piecewise linearly and locally flatly) embedded in the $4$-space $\mathbb R^4$. Two surface-links are said to be {\it equivalent} if they are ambient isotopic. By an {\it unoriented surface-link} we mean a non-orientable surface-link or an orientable surface-link without orientation, while an {\it oriented surface-link} means an orientable surface-link with a fixed orientation.

A {\it marked graph diagram} (or {\it ch-diagram}) is a link diagram possibly with some $4$-valent vertices equipped with markers; \xy (-4,4);(4,-4) **@{-}, 
(4,4);(-4,-4) **@{-},  
(3,-0.2);(-3,-0.2) **@{-},
(3,0);(-3,0) **@{-}, 
(3,0.2);(-3,0.2) **@{-}, 
\endxy. 
For a given marked graph diagram $D$, let $L_-(D)$ and $L_+(D)$ be classical link diagrams obtained from $D$ by replacing each marked vertex \xy (-4,4);(4,-4) **@{-}, 
(4,4);(-4,-4) **@{-},  
(3,-0.2);(-3,-0.2) **@{-},
(3,0);(-3,0) **@{-}, 
(3,0.2);(-3,0.2) **@{-}, 
\endxy with \xy (-4,4);(-4,-4) **\crv{(1,0)},  
(4,4);(4,-4) **\crv{(-1,0)}, 
\endxy and \xy (-4,4);(4,4) **\crv{(0,-1)}, 
(4,-4);(-4,-4) **\crv{(0,1)},   
\endxy, respectively (see Fig.~\ref{fig-nori-mg}). We call $L_-(D)$ and $L_+(D)$ the {\it negative resolution} and the {\it positive resolution} of D, respectively.
A marked graph diagram $D$ is said to be {\it admissible} if both resolutions $L_-(D)$ and $L_+(D)$ are trivial link diagrams.

\begin{figure}[ht]
\begin{center}
\resizebox{0.52\textwidth}{!}{%
  \includegraphics{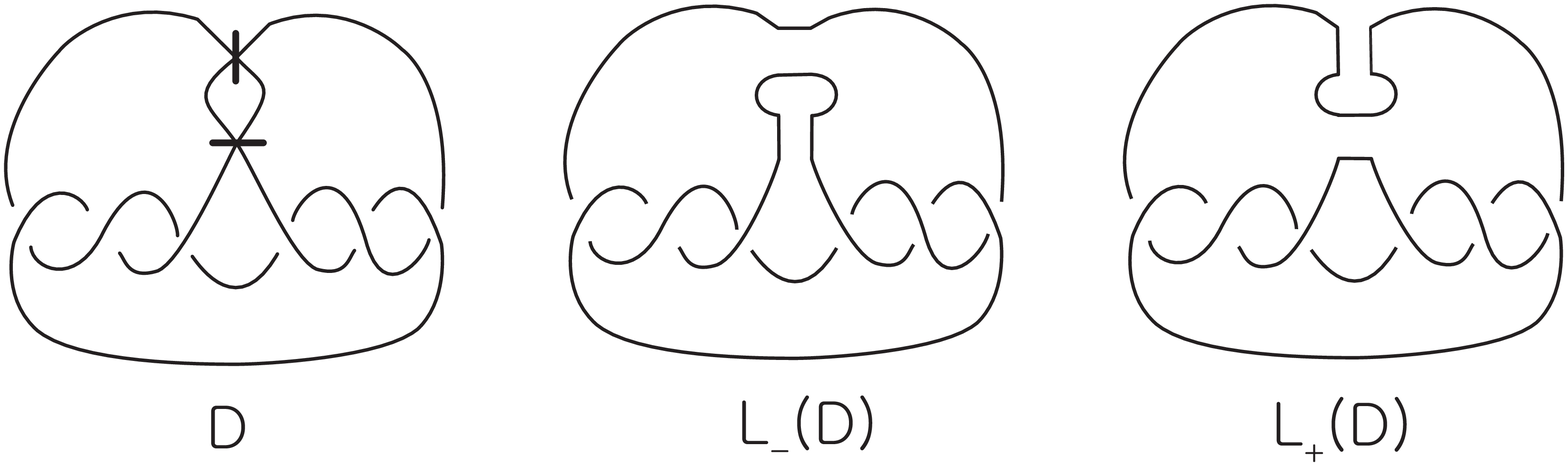}}
\caption{A marked graph diagram and its resolutions}\label{fig-nori-mg}
\end{center}
\end{figure}

S. J. Lomonaco, Jr. \cite{Lo} and K. Yoshikawa \cite{Yo} introduced a method of describing surface-links by marked graph diagrams. Indeed, every surface-link $\mathcal L$ is represented by an admissible marked graph diagram $D$. Moreover, if $D$ is an admissible marked graph diagram representing a surface-link $\mathcal L$, then we can construct a surface-link $\mathcal S(D)$ from $D$ so that $\mathcal S(D)$ is equivalent to $\mathcal L$. 
An {\it oriented marked graph diagram} is a marked graph diagram in which every edge has an orientation such that each marked vertex looks like
\xy (-4,4);(4,-4) **@{-}, 
(4,4);(-4,-4) **@{-}, 
(3,3.2)*{\llcorner}, 
(-3,-3.4)*{\urcorner}, 
(-2.5,2)*{\ulcorner},
(2.5,-2.4)*{\lrcorner}, 
(3,-0.2);(-3,-0.2) **@{-},
(3,0);(-3,0) **@{-}, 
(3,0.2);(-3,0.2) **@{-}, 
\endxy (see Fig.~\ref{fig-ori-mg}).
If $\mathcal L$ is an oriented surface-link, then it is represented by an admissible oriented marked graph diagram. See Section \ref{sect-omgr-osl} for details. We say that an (oriented) surface-link $\mathcal L$ is {\it presented} by an (oriented) marked graph diagram $D$ if $\mathcal L$ is ambient isotopic to the (oriented) surface-link $\mathcal S(D)$. 

\begin{figure}[ht]
\begin{center}
\resizebox{0.17\textwidth}{!}{%
  \includegraphics{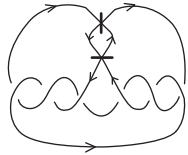}}
\caption{An oriented marked graph diagram}
\label{fig-ori-mg}
\end{center}
\end{figure}

In \cite{Yo}, K. Yoshikawa introduced local moves on marked graph diagrams, nowadays called Yoshikawa moves. In what follows the {\it unoriented Yoshikawa moves} mean the deformations $\Omega_1, \Omega_2, \Omega_3, \Omega_4, \Omega'_4, \Omega_5$(Type I) and $\Omega_6, \Omega'_6, \Omega_7, \Omega_8$ (Type II) on marked graph diagrams without orientations and their all mirror moves shown in Fig.~\ref{fig-moves-type-II}, Fig.~\ref{fig-r45678n} and Fig.~\ref{fig-rmove}. 
The {\it oriented Yoshikawa moves} mean the unoriented Yoshikawa moves with all possible orientations shown in  Fig.~\ref{fig-moves-type-II-o}, Fig.~\ref{fig-r123o} and Fig.~\ref{fig-om2}. Then we have

 \begin{theorem}[\cite{KK,Sw}]\label{thm-eqiv-ch-digs}
Let $D$ and $D'$ be admissible (oriented) marked graph diagrams and let $\mathcal {L}$ and $\mathcal {L}'$ be the unoriented (oriented) surface-links presented by $D$ and $D'$, respectively. Then the following statements are equivalent.
\begin{itemize}
  \item [(1)] $\mathcal {L}$ and $\mathcal {L}'$ are equivalent.
  \item [(2)] $D$ can be related to $D'$ by a finite sequence of unoriented (oriented) Yoshikawa moves.
\end{itemize}
\end{theorem} 
Using these terminologies, some properties and invariants of surface-links were studied in \cite{As,JKaL,JKL,JKL2,Le1,Le2,Le3,So,Yo}.

\begin{figure}
\begin{center}
\centerline{
\xy (12,2);(16,6) **@{-},
(12,6);(13.5,4.5) **@{-},
(14.5,3.5);(16,2) **@{-},
(16,6);(22,6) **\crv{(18,8)&(20,8)},
(16,2);(22,2) **\crv{(18,0)&(20,0)}, (22,6);(22,2) **\crv{(23.5,4)},
(7,8);(12,6) **\crv{(10,8)}, (7,0);(12,2) **\crv{(10,0)},
(35,5);(45,5) **@{-} ?>*\dir{>}, (35,3);(45,3) **@{-} ?<*\dir{<},
(57,8);(57,0) **\crv{(67,7)&(67,1)}, (-5,4)*{\Omega_1 :}, (73,4)*{},
\endxy}

\vskip.3cm

\centerline{ \xy (7,7);(7,1)  **\crv{(23,6)&(23,2)}, (16,6.3);(23,7)
**\crv{(19,6.9)}, (16,1.7);(23,1) **\crv{(19,1.1)},
(14,5.7);(14,2.3) **\crv{(8,4)},
(35,5);(45,5) **@{-} ?>*\dir{>}, (35,3);(45,3) **@{-} ?<*\dir{<},
(57,7);(57,1) **\crv{(65,6)&(65,2)}, (73,7);(73,1)
**\crv{(65,6)&(65,2)}, (-5,4)*{\Omega_2 :},
\endxy}

\vskip.3cm

\centerline{ \xy (7,6);(23,6)  **\crv{(15,-2)},
(10,0);(11.5,1.8) **@{-},
(12.5,3);(14.5,5.5) **@{-},
(15.5,6.6);(20,12) **@{-},
(10,12);(17.5,3) **@{-},
(18.5,1.8);(20,0) **@{-},
(35,7);(45,7) **@{-} ?>*\dir{>},
(35,5);(45,5) **@{-} ?<*\dir{<},
(57,6);(73,6)  **\crv{(65,14)},
(70,12);(68.5,10.2) **@{-},
(62.5,9);(70,0) **@{-},
(67.5,9);(65.6,6.6) **@{-},
(64.5,5.3);(60,0) **@{-},
(61.5,10.2);(60,12) **@{-},
(-5,6)*{\Omega_3 :},
\endxy}

\vskip.3cm

 \centerline{ \xy
 (7,6);(23,6)  **\crv{(15,-2)},
 (10,0);(11.5,1.8) **@{-},
(12.5,3);(20,12) **@{-},
(10,12);(17.5,3) **@{-},
(18.5,1.8);(20,0) **@{-},
(15,6);(17,6) **@{-},
(13,6.1);(17,6.1) **@{-},
(13,5.9);(17,5.9) **@{-},
(13,6.2);(17,6.2) **@{-},
(13,5.8);(17,5.8) **@{-},
(35,7);(45,7) **@{-} ?>*\dir{>},
(35,5);(45,5) **@{-} ?<*\dir{<},
(57,6);(73,6)  **\crv{(65,14)},
(70,12);(68.5,10.2) **@{-},
(67.5,9);(60,0) **@{-},
(70,0);(62.5,9) **@{-},
(61.5,10.2);(60,12) **@{-},
(63,6.1);(67,6.1) **@{-},
(63,5.9);(67,5.9) **@{-},
(63,6.2);(67,6.2) **@{-},
(63,5.8);(67,5.8) **@{-},
(-5,6)*{\Omega_4:},
\endxy}

\vskip.3cm

 \centerline{ \xy 
  (13,2.2);(17,2.2)  **\crv{(15,1.7)}, 
  (7,6);(11,3)  **\crv{(10,3.5)}, 
  (23,6);(19,3)  **\crv{(20,3.5)}, 
 (10,0);(20,12) **@{-}, 
(10,12);(20,0) **@{-}, 
(13,6);(17,6) **@{-}, (13,6.1);(17,6.1) **@{-}, (13,5.9);(17,5.9)
**@{-}, (13,6.2);(17,6.2) **@{-}, (13,5.8);(17,5.8) **@{-}, 
(35,7);(45,7) **@{-} ?>*\dir{>}, 
(35,5);(45,5) **@{-} ?<*\dir{<},
   (63,9.8);(67,9.8)  **\crv{(65,10.3)}, 
  (57,6);(61,9)  **\crv{(60,8.5)}, 
  (73,6);(69,9)  **\crv{(70,8.5)}, 
(70,12);(60,0) **@{-}, 
(70,0);(60,12) **@{-}, 
(63,6);(67,6) **@{-}, 
(63,6.1);(67,6.1) **@{-}, 
(63,5.9);(67,5.9) **@{-},
(63,6.2);(67,6.2) **@{-}, 
(63,5.8);(67,5.8) **@{-}, 
(-5,6)*{\Omega'_4:},
\endxy}

\vskip.3cm

 \centerline{ \xy (9,2);(13,6) **@{-}, (9,6);(10.5,4.5) **@{-},
(11.5,3.5);(13,2) **@{-}, (17,2);(21,6) **@{-}, (17,6);(21,2)
**@{-}, (13,6);(17,6) **\crv{(15,8)}, (13,2);(17,2) **\crv{(15,0)},
(7,7);(9,6) **\crv{(8,7)}, (7,1);(9,2) **\crv{(8,1)}, (23,7);(21,6)
**\crv{(22,7)}, (23,1);(21,2) **\crv{(22,1)},
(17,4);(21,4) **@{-},
(17,4.1);(21,4.1) **@{-},
(17,3.9);(21,3.9) **@{-},
(17,4.2);(21,4.2) **@{-},
(17,3.8);(21,3.8) **@{-},
(35,5);(45,5) **@{-} ?>*\dir{>}, (35,3);(45,3) **@{-} ?<*\dir{<},
(59,2);(63,6) **@{-}, (59,6);(63,2) **@{-}, (67,2);(71,6) **@{-},
(67,6);(68.5,4.5) **@{-}, (69.5,3.5);(71,2) **@{-}, (63,6);(67,6)
**\crv{(65,8)}, (63,2);(67,2) **\crv{(65,0)}, (57,7);(59,6)
**\crv{(58,7)}, (57,1);(59,2) **\crv{(58,1)}, (73,7);(71,6)
**\crv{(72,7)}, (73,1);(71,2) **\crv{(72,1)},
(59,4);(63,4) **@{-},
(59,4.1);(63,4.1) **@{-}, (59,3.9);(63,3.9) **@{-},
(59,4.2);(63,4.2) **@{-}, (59,3.8);(63,3.8) **@{-}, 
(-5,4)*{\Omega_5:},
\endxy}

\vskip 0.3cm

\centerline{ \xy (12,6);(16,2) **@{-}, (12,2);(16,6) **@{-},
(16,6);(22,6) **\crv{(18,8)&(20,8)}, (16,2);(22,2)
**\crv{(18,0)&(20,0)}, (22,6);(22,2) **\crv{(23.5,4)}, (7,8);(12,6)
**\crv{(10,8)}, (7,0);(12,2) **\crv{(10,0)},
(35,5);(45,5) **@{-} ?>*\dir{>}, (35,3);(45,3) **@{-} ?<*\dir{<},
(57,8);(57,0) **\crv{(67,7)&(67,1)}, (-5,4)*{\Omega_6 :}, (73,4)*{},
(14,6);(14,2) **@{-}, (14.1,6);(14.1,2) **@{-}, (13.9,6);(13.9,2)
**@{-}, (14.2,6);(14.2,2) **@{-}, (13.8,6);(13.8,2) **@{-},
\endxy}

\vskip.3cm

\centerline{ \xy (12,6);(16,2) **@{-}, (12,2);(16,6) **@{-},
(16,6);(22,6) **\crv{(18,8)&(20,8)}, (16,2);(22,2)
**\crv{(18,0)&(20,0)}, (22,6);(22,2) **\crv{(23.5,4)}, (7,8);(12,6)
**\crv{(10,8)}, (7,0);(12,2) **\crv{(10,0)},
(35,5);(45,5) **@{-} ?>*\dir{>}, (35,3);(45,3) **@{-} ?<*\dir{<},
(57,8);(57,0) **\crv{(67,7)&(67,1)}, (-5,4)*{\Omega'_6 :},
(73,4)*{}, (12,4);(16,4) **@{-}, (12,4.1);(16,4.1) **@{-},
(12,4.2);(16,4.2) **@{-}, (12,3.9);(16,3.9) **@{-},
(12,3.8);(16,3.8) **@{-},
\endxy}

\vskip.3cm

\centerline{ \xy (9,4);(17,12) **@{-}, (9,8);(13,4) **@{-},
(17,12);(21,16) **@{-}, (17,16);(21,12) **@{-}, (7,0);(9,4)
**\crv{(7,2)}, (7,12);(9,8) **\crv{(7,10)}, (15,0);(13,4)
**\crv{(15,2)}, (17,16);(15,20) **\crv{(15,18)}, (21,16);(23,20)
**\crv{(23,18)}, (21,12);(23,8) **\crv{(23,10)}, (7,12);(7,20)
**@{-}, (23,8);(23,0) **@{-},
(11,4);(11,8) **@{-},
(10.9,4);(10.9,8) **@{-},
(11.1,4);(11.1,8) **@{-},
(10.8,4);(10.8,8) **@{-},
(11.2,4);(11.2,8) **@{-},
(17,14);(21,14) **@{-},
(17,14.1);(21,14.1) **@{-},
(17,13.9);(21,13.9) **@{-},
(17,14.2);(21,14.2) **@{-},
(17,13.8);(21,13.8) **@{-},
(35,11);(45,11) **@{-} ?>*\dir{>}, (35,9);(45,9) **@{-} ?<*\dir{<},
(71,4);(63,12) **@{-}, (71,8);(67,4) **@{-}, (63,12);(59,16) **@{-},
(63,16);(59,12) **@{-}, (73,0);(71,4) **\crv{(73,2)}, (73,12);(71,8)
**\crv{(73,10)}, (65,0);(67,4) **\crv{(65,2)}, (63,16);(65,20)
**\crv{(65,18)}, (59,16);(57,20) **\crv{(57,18)}, (59,12);(57,8)
**\crv{(57,10)}, (73,12);(73,20) **@{-}, (57,8);(57,0) **@{-},
(61,12);(61,16) **@{-},
(61.1,12);(61.1,16) **@{-},
(60.9,12);(60.9,16) **@{-},
(61.2,12);(61.2,16) **@{-},
(60.8,12);(60.8,16) **@{-},
(67,6);(71,6) **@{-},
(67,6.1);(71,6.1) **@{-},
(67,5.9);(71,5.9) **@{-},
(67,6.2);(71,6.2) **@{-},
(67,5.8);(71,5.8) **@{-},
(-5,10)*{\Omega_7:},
 \endxy}

\vskip.3cm

\centerline{ \xy (7,20);(14.2,11) **@{-}, (15.8,9);(17.4,7) **@{-},
(19,5);(23,0) **@{-}, (13,20);(7,12) **@{-}, (7,12);(11.2,7) **@{-},
(12.7,5.2);(14.4,3.2) **@{-}, (15.7,1.6);(17,0) **@{-},
(17,20);(23,12) **@{-}, (13,0);(23,12) **@{-}, (7,0);(23,20) **@{-},
(10,18);(10,14) **@{-}, (10.1,18);(10.1,14) **@{-},
(9.9,18);(9.9,14) **@{-}, (10.2,18);(10.2,14) **@{-},
(9.8,18);(9.8,14) **@{-}, (18,16);(22,16) **@{-},
(18,16.1);(22,16.1) **@{-}, (18,15.9);(22,15.9) **@{-},
(18,16.2);(22,16.2) **@{-}, (18,15.8);(22,15.8) **@{-},
(35,11);(45,11) **@{-} ?>*\dir{>}, (35,9);(45,9) **@{-} ?<*\dir{<},
(73,20);(65.8,11) **@{-}, (64.2,9);(62.6,7) **@{-}, (61,5);(57,0)
**@{-}, (67,20);(73,12) **@{-}, (73,12);(68.8,7) **@{-},
(67.3,5.2);(65.6,3.2) **@{-}, (64.3,1.6);(63,0) **@{-},
(63,20);(57,12) **@{-}, (67,0);(57,12) **@{-}, (73,0);(57,20)
**@{-},
(60,18);(60,14) **@{-}, (60.1,18);(60.1,14) **@{-},
(59.9,18);(59.9,14) **@{-}, (60.2,18);(60.2,14) **@{-},
(59.8,18);(59.8,14) **@{-}, (68,16);(72,16) **@{-},
(68,16.1);(72,16.1) **@{-}, (68,15.9);(72,15.9) **@{-},
(68,16.2);(72,16.2) **@{-}, (68,15.8);(72,15.8) **@{-},
(-5,10)*{\Omega_8:},
\endxy}
\caption{A generating set of unoriented Yoshikawa moves}
\label{fig-moves-type-II}
\quad
 \centerline{ \xy 
 (7,6);(23,6)  **\crv{(15,-2)}, 
 (10,0);(11.5,1.8) **@{-},
(12.5,3);(20,12) **@{-}, 
(10,12);(17.5,3) **@{-}, 
(18.5,1.8);(20,0) **@{-}, 
(15,4);(15,8) **@{-},
(14.9,4);(14.9,8) **@{-},
(15.1,4);(15.1,8) **@{-},
(14.8,4);(14.8,8) **@{-},
(15.2,4);(15.2,8) **@{-},
(25,7);(30,7) **@{-} ?>*\dir{>} ?<*\dir{<}, (32,7) *{},
(28,3)*{\Omega_{4a}},
\endxy
\xy (57,6);(73,6)  **\crv{(65,14)}, 
(70,12);(68.5,10.2) **@{-},
(67.5,9);(60,0) **@{-}, 
(70,0);(62.5,9) **@{-},
(61.5,10.2);(60,12) **@{-},  
(65,4);(65,8) **@{-},
(64.9,4);(64.9,8) **@{-},
(65.1,4);(65.1,8) **@{-},
(64.8,4);(64.8,8) **@{-},
(65.2,4);(65.2,8) **@{-},
\endxy
\qquad
\xy 
(7,6);(11,3)  **\crv{(10,3.5)}, 
(23,6);(19,3)  **\crv{(20,3.5)}, 
(13,2);(17,2)  **\crv{(15,1)}, 
(10,0);(20,12) **@{-},
(10,12);(20,0) **@{-}, 
(15,4);(15,8) **@{-},
(14.9,4);(14.9,8) **@{-},
(15.1,4);(15.1,8) **@{-},
(14.8,4);(14.8,8) **@{-},
(15.2,4);(15.2,8) **@{-},
 (25,7);(30,7) **@{-} ?>*\dir{>} ?<*\dir{<}, (32,7) *{},
 (28,3)*{\Omega'_{4a}},
 \endxy 
\xy (57,6);(61,9)  **\crv{(60,8.5)}, 
(73,6);(69,9)  **\crv{(70,8.5)}, 
(63,10);(67,10)  **\crv{(65,11)}, 
(70,12);(60,0) **@{-}, 
(70,0);(60,12) **@{-},  
(65,4);(65,8) **@{-},
(64.9,4);(64.9,8) **@{-},
(65.1,4);(65.1,8) **@{-},
(64.8,4);(64.8,8) **@{-},
(65.2,4);(65.2,8) **@{-},
\endxy}

\vskip.3cm

\centerline{ 
\xy (9,2);(13,6) **@{-}, (9,6);(10.5,4.5) **@{-},
(11.5,3.5);(13,2) **@{-}, (17,2);(21,6) **@{-}, (17,6);(21,2)
**@{-}, (13,6);(17,6) **\crv{(15,8)}, (13,2);(17,2) **\crv{(15,0)},
(7,7);(9,6) **\crv{(8,7)}, (7,1);(9,2) **\crv{(8,1)}, (23,7);(21,6)
**\crv{(22,7)}, (23,1);(21,2) **\crv{(22,1)}, 
(19,2);(19,6) **@{-},
(19.1,2);(19.1,6) **@{-},
(18.9,2);(18.9,6) **@{-},
(19.2,2);(19.2,6) **@{-},
(18.8,2);(18.8,6) **@{-},
 (25,3.3);(30,3.3) **@{-} ?>*\dir{>} ?<*\dir{<}, (32,7) *{},
 (28,0)*{\Omega_{5a}},
\endxy
\xy (59,2);(63,6) **@{-}, (59,6);(63,2) **@{-}, (67,2);(71,6) **@{-},
(67,6);(68.5,4.5) **@{-}, (69.5,3.5);(71,2) **@{-}, (63,6);(67,6)
**\crv{(65,8)}, (63,2);(67,2) **\crv{(65,0)}, (57,7);(59,6)
**\crv{(58,7)}, (57,1);(59,2) **\crv{(58,1)}, (73,7);(71,6)
**\crv{(72,7)}, (73,1);(71,2) **\crv{(72,1)}, 
(61,0) *{}, 
(61,2);(61,6) **@{-},
(61.1,2);(61.1,6) **@{-},
(60.9,2);(60.9,6) **@{-},
(61.2,2);(61.2,6) **@{-},
(60.8,2);(60.8,6) **@{-},
\endxy
\quad
\xy (11.7,4.7);(13,6) **@{-},  
(9,2);(10.6,3.4) **@{-}, 
(9,6);(12.5,2.5) **@{-},
(11.5,3.5);(13,2) **@{-}, 
(17,2);(21,6) **@{-}, 
(17,6);(21,2) **@{-}, 
(13,6);(17,6) **\crv{(15,8)}, 
(13,2);(17,2) **\crv{(15,0)},
(7,7);(9,6) **\crv{(8,7)}, (7,1);(9,2) **\crv{(8,1)}, (23,7);(21,6)
**\crv{(22,7)}, (23,1);(21,2) **\crv{(22,1)}, 
(17,4);(21,4) **@{-},
(17,4.1);(21,4.1) **@{-}, 
(17,3.9);(21,3.9) **@{-},
(17,4.2);(21,4.2) **@{-}, 
(17,3.8);(21,3.8) **@{-},
 (25,3.3);(30,3.3) **@{-} ?>*\dir{>} ?<*\dir{<}, 
 (28,0)*{\Omega_{5b}},
\endxy
\xy (59,2);(63,6) **@{-}, 
(59,6);(63,2) **@{-}, 
(67,2);(68.3,3.3) **@{-},
(69.5,4.7);(71,6) **@{-},
(67,6);(69.5,3.5) **@{-}, 
(69.5,3.5);(71,2) **@{-}, 
(63,6);(67,6) **\crv{(65,8)}, 
(63,2);(67,2) **\crv{(65,0)}, 
(57,7);(59,6) **\crv{(58,7)}, 
(57,1);(59,2) **\crv{(58,1)}, 
(73,7);(71,6) **\crv{(72,7)}, 
(73,1);(71,2) **\crv{(72,1)}, 
(59,4);(63,4) **@{-},
(59,4.1);(63,4.1) **@{-}, 
(59,3.9);(63,3.9) **@{-},
(59,4.2);(63,4.2) **@{-}, 
(59,3.8);(63,3.8) **@{-},
(61,0) *{}, 
\endxy
\quad
\xy (11.7,4.7);(13,6) **@{-},  
(9,2);(10.6,3.4) **@{-}, 
(9,6);(12.5,2.5) **@{-},
(11.5,3.5);(13,2) **@{-}, 
(17,2);(21,6) **@{-}, 
(17,6);(21,2) **@{-}, 
(13,6);(17,6) **\crv{(15,8)}, (13,2);(17,2) **\crv{(15,0)},
(7,7);(9,6) **\crv{(8,7)}, (7,1);(9,2) **\crv{(8,1)}, (23,7);(21,6)
**\crv{(22,7)}, (23,1);(21,2) **\crv{(22,1)}, 
(19,2);(19,6) **@{-},
(19.1,2);(19.1,6) **@{-},
(18.9,2);(18.9,6) **@{-},
(19.2,2);(19.2,6) **@{-},
(18.8,2);(18.8,6) **@{-},
 (25,3.3);(30,3.3) **@{-} ?>*\dir{>} ?<*\dir{<}, (32,7) *{},
 (28,0)*{\Omega_{5c}},
\endxy
\xy (59,2);(63,6) **@{-}, 
(59,6);(63,2) **@{-}, 
(67,2);(68.3,3.3) **@{-},
(69.5,4.7);(71,6) **@{-},
(67,6);(69.5,3.5) **@{-}, 
(69.5,3.5);(71,2) **@{-},  
(63,6);(67,6) **\crv{(65,8)}, 
(63,2);(67,2) **\crv{(65,0)}, 
(57,7);(59,6) **\crv{(58,7)}, 
(57,1);(59,2) **\crv{(58,1)}, 
(73,7);(71,6) **\crv{(72,7)}, 
(73,1);(71,2) **\crv{(72,1)}, 
(61,0) *{}, 
(61,2);(61,6) **@{-},
(61.1,2);(61.1,6) **@{-},
(60.9,2);(60.9,6) **@{-},
(61.2,2);(61.2,6) **@{-},
(60.8,2);(60.8,6) **@{-},
\endxy}

\vskip.2cm

\centerline{ 
\xy (7,20);(14.2,11) **@{-}, (15.8,9);(17.4,7) **@{-},
(19,5);(23,0) **@{-}, (13,20);(7,12) **@{-}, (7,12);(11.2,7) **@{-},
(12.7,5.2);(14.4,3.2) **@{-}, (15.7,1.6);(17,0) **@{-},
(17,20);(23,12) **@{-}, (13,0);(23,12) **@{-}, (7,0);(23,20) **@{-},
(20,18);(20,14) **@{-}, 
(20.1,18);(20.1,14) **@{-},
(19.9,18);(19.9,14) **@{-}, 
(20.2,18);(20.2,14) **@{-},
(19.8,18);(19.8,14) **@{-},
(8,16);(12,16) **@{-},
(8,16.1);(12,16.1) **@{-}, 
(8,15.9);(12,15.9) **@{-},
(8,16.2);(12,16.2) **@{-}, 
(8,15.8);(12,15.8) **@{-},
(27,8.3);(32,8.3) **@{-} ?>*\dir{>} ?<*\dir{<}, (37,12) *{},(30,5)*{\Omega_{8a}},
 \endxy
\xy (73,20);(65.8,11) **@{-}, (64.2,9);(62.6,7) **@{-}, (61,5);(57,0)
**@{-}, (67,20);(73,12) **@{-}, (73,12);(68.8,7) **@{-},
(67.3,5.2);(65.6,3.2) **@{-}, (64.3,1.6);(63,0) **@{-},
(63,20);(57,12) **@{-}, (67,0);(57,12) **@{-}, (73,0);(57,20)
**@{-},
(70,18);(70,14) **@{-}, 
(70.1,18);(70.1,14) **@{-},
(69.9,18);(69.9,14) **@{-}, 
(70.2,18);(70.2,14) **@{-},
(69.8,18);(69.8,14) **@{-}, 
(58,16);(62,16) **@{-},
(58,16.1);(62,16.1) **@{-}, 
(58,15.9);(62,15.9) **@{-},
(58,16.2);(62,16.2) **@{-}, 
(58,15.8);(62,15.8) **@{-},
\endxy}
\caption{Unoriented Yoshikawa moves}
\label{fig-r45678n}
\end{center}
\end{figure}

\begin{figure}
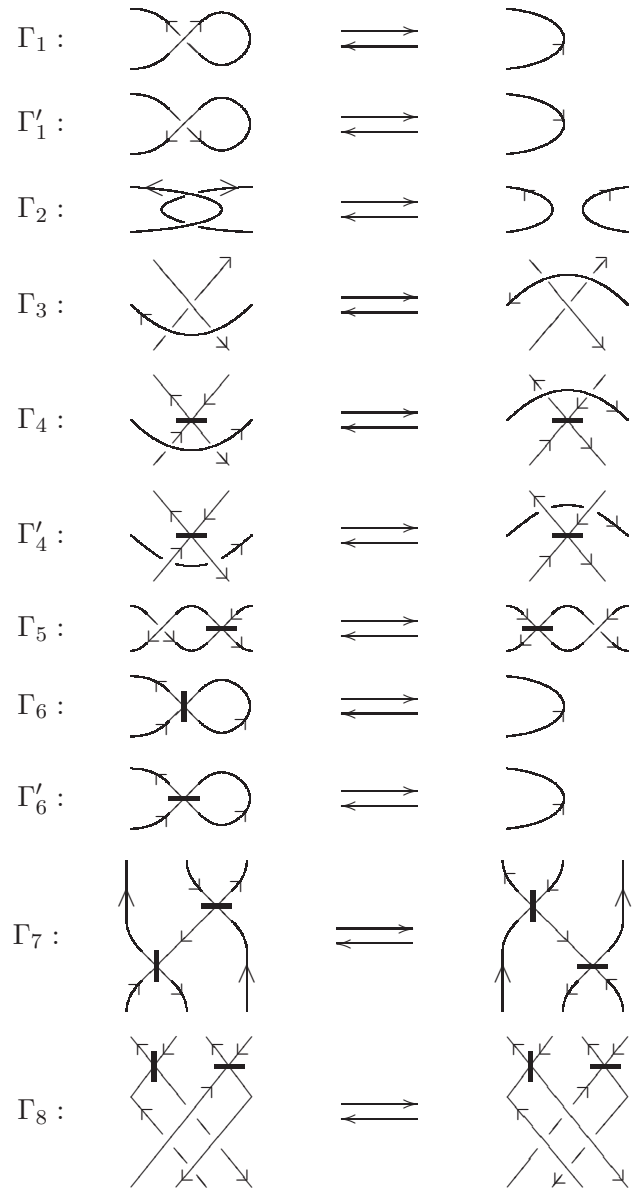

\begin{center}
\centerline{ 
\xy (12,2);(16,6) **@{-}, 
(12,6);(13.5,4.5) **@{-},
(14.5,3.5);(16,2) **@{-}, 
(16,6);(22,6) **\crv{(18,8)&(20,8)},
(16,2);(22,2) **\crv{(18,0)&(20,0)}, (22,6);(22,2) **\crv{(23.5,4)},
(7,8);(12,6) **\crv{(10,8)}, (7,0);(12,2) **\crv{(10,0)}, (12.4,5) *{\ulcorner}, (15.6,5) *{\urcorner},
(35,5);(45,5) **@{-} ?>*\dir{>}, (35,3);(45,3) **@{-} ?<*\dir{<},
(63.8,2) *{\urcorner},
(57,8);(57,0) **\crv{(67,7)&(67,1)}, (-5,4)*{\Gamma_1 :}, (73,4)*{},
\endxy }

\vskip.3cm

\centerline{ 
\xy (12,2);(16,6) **@{-}, 
(12,6);(13.5,4.5) **@{-},
(14.5,3.5);(16,2) **@{-}, 
(16,6);(22,6) **\crv{(18,8)&(20,8)},
(16,2);(22,2) **\crv{(18,0)&(20,0)}, (22,6);(22,2) **\crv{(23.5,4)},
(7,8);(12,6) **\crv{(10,8)}, (7,0);(12,2) **\crv{(10,0)}, (12.4,2.5) *{\llcorner}, (15.6,2.5) *{\lrcorner},
(35,5);(45,5) **@{-} ?>*\dir{>}, (35,3);(45,3) **@{-} ?<*\dir{<},(63.85,5.2) *{\lrcorner},
(57,8);(57,0) **\crv{(67,7)&(67,1)}, (-5,4)*{\Gamma'_1 :}, (73,4)*{},
\endxy }

\vskip.3cm

\centerline{ \xy (7,7);(7,1)  **\crv{(23,6)&(23,2)}, (16,6.3);(23,7)
**\crv{(19,6.9)}, (16,1.7);(23,1) **\crv{(19,1.1)},
(14,5.7);(14,2.3) **\crv{(8,4)}, (10,6.9) *{<}, (20,6.9) *{>},
(35,5);(45,5) **@{-} ?>*\dir{>}, (35,3);(45,3) **@{-} ?<*\dir{<},
(57,7);(57,1) **\crv{(65,6)&(65,2)}, (73,7);(73,1)
**\crv{(65,6)&(65,2)}, (60,5.3) *{\ulcorner}, (70,5.3) *{\urcorner}, (-5,4)*{\Gamma_2 :},
\endxy}

\vskip.3cm

\centerline{ 
\xy (7,6);(23,6) **\crv{(15,-2)}, 
(10,0);(11.5,1.8) **@{-}, 
(17.5,3);(14.5,6.6) **@{-},
(14.5,6.6);(10,12) **@{-}, 
(20,12);(15.5,6.6) **@{-},
(14.5,5.5);(12.5,3) **@{-},
(18.5,1.8);(20,0) **@{-},
(19.5,11) *{\urcorner}, 
(19,1.1) *{\lrcorner},  
(9,3.5) *{\ulcorner},
(35,7);(45,7) **@{-} ?>*\dir{>}, 
(35,5);(45,5) **@{-} ?<*\dir{<},
(57,6);(73,6) **\crv{(65,14)}, 
(70,12);(68.5,10.2) **@{-}, 
(67.5,9);(65.5,6.5) **@{-}, 
(64.6,5.5);(60,0) **@{-}, 
(62.5,9);(64.4,6.6) **@{-}, 
(64.4,6.6);(70,0) **@{-}, 
(61.5,10.2);(60,12) **@{-},
(69.5,11) *{\urcorner}, 
(69,1.1) *{\lrcorner},  
(58,7) *{\llcorner},
(-5,6)*{\Gamma_3:},
\endxy}

\vskip.3cm

 \centerline{ \xy 
 (7,6);(23,6)  **\crv{(15,-2)}, 
 (10,0);(11.5,1.8) **@{-},
(12.5,3);(20,12) **@{-}, 
(10,12);(17.5,3) **@{-}, 
(18.5,1.8);(20,0) **@{-}, 
(13,6);(17,6) **@{-}, (13,6.1);(17,6.1) **@{-}, (13,5.9);(17,5.9)
**@{-}, (13,6.2);(17,6.2) **@{-}, (13,5.8);(17,5.8) **@{-},
  (13,3.1) *{\urcorner}, (17.5,8.9) *{\llcorner}, (19,1.1) *{\lrcorner},  (21,3.5) *{\urcorner},(13,8) *{\ulcorner},
(35,7);(45,7) **@{-} ?>*\dir{>}, 
(35,5);(45,5) **@{-} ?<*\dir{<},
(57,6);(73,6)  **\crv{(65,14)}, 
(70,12);(68.5,10.2) **@{-},
(67.5,9);(60,0) **@{-}, 
(70,0);(62.5,9) **@{-}, 
(61.5,10.2);(60,12) **@{-}, 
(63,6);(67,6) **@{-}, (63,6.1);(67,6.1) **@{-}, (63,5.9);(67,5.9)
**@{-}, (63,6.2);(67,6.2) **@{-}, (63,5.8);(67,5.8) **@{-},
(62,2) *{\urcorner}, (67,8.3) *{\llcorner}, (67.5,3) *{\lrcorner},  (70.9,8) *{\lrcorner},(61.3,10) *{\ulcorner}, 
(-5,6)*{\Gamma_4:},
\endxy}

\vskip.3cm

 \centerline{ \xy 
  (13,2.2);(17,2.2)  **\crv{(15,1.7)}, 
  (7,6);(11,3)  **\crv{(10,3.5)}, 
  (23,6);(19,3)  **\crv{(20,3.5)}, 
 (10,0);(20,12) **@{-}, 
(10,12);(20,0) **@{-}, 
(13,6);(17,6) **@{-}, (13,6.1);(17,6.1) **@{-}, (13,5.9);(17,5.9)
**@{-}, (13,6.2);(17,6.2) **@{-}, (13,5.8);(17,5.8) **@{-}, 
(13,3.1) *{\urcorner}, (17.5,8.9) *{\llcorner}, (19,1.1) *{\lrcorner},  (21,3.5) *{\urcorner},(13,8) *{\ulcorner},
(35,7);(45,7) **@{-} ?>*\dir{>}, 
(35,5);(45,5) **@{-} ?<*\dir{<},
   (63,9.8);(67,9.8)  **\crv{(65,10.3)}, 
  (57,6);(61,9)  **\crv{(60,8.5)}, 
  (73,6);(69,9)  **\crv{(70,8.5)}, 
(70,12);(60,0) **@{-}, 
(70,0);(60,12) **@{-}, 
(63,6);(67,6) **@{-}, (63,6.1);(67,6.1) **@{-}, (63,5.9);(67,5.9)
**@{-},(63,6.2);(67,6.2) **@{-}, (63,5.8);(67,5.8) **@{-}, 
(62,2) *{\urcorner}, (67,8.3) *{\llcorner}, (67.5,3) *{\lrcorner},  (70.9,8) *{\lrcorner},(61.3,10) *{\ulcorner}, 
(-5,6)*{\Gamma'_4:},
\endxy}

\vskip.3cm

\centerline{ \xy (9,2);(13,6) **@{-}, (9,6);(10.5,4.5) **@{-},
(11.5,3.5);(13,2) **@{-}, (17,2);(21,6) **@{-}, (17,6);(21,2)
**@{-}, (13,6);(17,6) **\crv{(15,8)}, (13,2);(17,2) **\crv{(15,0)},
(7,7);(9,6) **\crv{(8,7)}, (7,1);(9,2) **\crv{(8,1)}, (23,7);(21,6)
**\crv{(22,7)}, (23,1);(21,2) **\crv{(22,1)}, 
(17,4);(21,4) **@{-}, (17,4.1);(21,4.1) **@{-}, (17,3.9);(21,3.9)
**@{-}, (17,4.2);(21,4.2) **@{-}, (17,3.8);(21,3.8) **@{-},
(10,3) *{\llcorner},  (12,3) *{\lrcorner}, (21,6) *{\llcorner},(21,2.2) *{\lrcorner},
(35,5);(45,5) **@{-} ?>*\dir{>}, (35,3);(45,3) **@{-} ?<*\dir{<},
(59,2);(63,6) **@{-}, (59,6);(63,2) **@{-}, (67,2);(71,6) **@{-},
(67,6);(68.5,4.5) **@{-}, (69.5,3.5);(71,2) **@{-}, (63,6);(67,6)
**\crv{(65,8)}, (63,2);(67,2) **\crv{(65,0)}, (57,7);(59,6)
**\crv{(58,7)}, (57,1);(59,2) **\crv{(58,1)}, (73,7);(71,6)
**\crv{(72,7)}, (73,1);(71,2) **\crv{(72,1)}, 
(63,4);(59,4) **@{-}, (63,4.1);(59,4.1) **@{-}, (63,3.9);(59,3.9)
**@{-}, (63,4.2);(59,4.2) **@{-}, (63,3.8);(59,3.8) **@{-},
(59.5,2.5) *{\llcorner},  (59,6) *{\lrcorner}, (71,6) *{\llcorner},(71,2.2) *{\lrcorner},
 (-5,4)*{\Gamma_5:},
\endxy}

\vskip.3cm

\centerline{ 
\xy (12,6);(16,2) **@{-}, (12,2);(16,6) **@{-},
(16,6);(22,6) **\crv{(18,8)&(20,8)}, (16,2);(22,2)
**\crv{(18,0)&(20,0)}, (22,6);(22,2) **\crv{(23.5,4)}, (7,8);(12,6)
**\crv{(10,8)}, (7,0);(12,2) **\crv{(10,0)}, (11,0.4) *{\urcorner}, (11,6) *{\ulcorner},(21.5,1)*{\urcorner},
(35,5);(45,5) **@{-} ?>*\dir{>}, (35,3);(45,3) **@{-} ?<*\dir{<},
(57,8);(57,0) **\crv{(67,7)&(67,1)}, (-5,4)*{\Gamma_6 :}, (73,4)*{},
(14,6);(14,2) **@{-}, (14.1,6);(14.1,2) **@{-}, (13.9,6);(13.9,2)
**@{-}, (14.2,6);(14.2,2) **@{-}, (13.8,6);(13.8,2) **@{-}, 
(63.8,2) *{\urcorner},
\endxy}

\vskip.3cm

\centerline{ \xy (12,6);(16,2) **@{-}, (12,2);(16,6) **@{-},
(16,6);(22,6) **\crv{(18,8)&(20,8)}, (16,2);(22,2)
**\crv{(18,0)&(20,0)}, (22,6);(22,2) **\crv{(23.5,4)}, (7,8);(12,6)
**\crv{(10,8)}, (7,0);(12,2) **\crv{(10,0)}, (11,0.4) *{\urcorner}, (11,6) *{\ulcorner},(21.5,1)*{\urcorner},
(35,5);(45,5) **@{-} ?>*\dir{>}, (35,3);(45,3) **@{-} ?<*\dir{<},
(57,8);(57,0) **\crv{(67,7)&(67,1)}, (-5,4)*{\Gamma'_6 :},
(73,4)*{}, (12,4);(16,4) **@{-}, (12,4.1);(16,4.1) **@{-},
(12,4.2);(16,4.2) **@{-}, (12,3.9);(16,3.9) **@{-},
(12,3.8);(16,3.8) **@{-}, (63.8,2) *{\urcorner},
\endxy}

\vskip.3cm

\centerline{ \xy (9,4);(17,12) **@{-}, (9,8);(13,4) **@{-},
(17,12);(21,16) **@{-}, (17,16);(21,12) **@{-}, (7,0);(9,4)
**\crv{(7,2)}, (7,12);(9,8) **\crv{(7,10)}, (15,0);(13,4)
**\crv{(15,2)}, (17,16);(15,20) **\crv{(15,18)}, (21,16);(23,20)
**\crv{(23,18)}, (21,12);(23,8) **\crv{(23,10)}, (7,12);(7,20)
**@{-}, (23,8);(23,0) **@{-},
(11,4);(11,8) **@{-}, 
(10.9,4);(10.9,8) **@{-}, 
(11.1,4);(11.1,8) **@{-}, 
(10.8,4);(10.8,8) **@{-}, 
(11.2,4);(11.2,8) **@{-},
(17,14);(21,14) **@{-}, 
(17,14.1);(21,14.1) **@{-},
(17,13.9);(21,13.9) **@{-}, 
(17,14.2);(21,14.2) **@{-},
(17,13.8);(21,13.8) **@{-},
(7,15) *{\wedge},(23,5) *{\wedge},(15,10) *{\llcorner},(8,2.3) *{\urcorner}, (21.5,16) *{\urcorner},(16,17) *{\lrcorner},(13.7,3) *{\lrcorner},
(35,11);(45,11) **@{-} ?>*\dir{>}, (35,9);(45,9) **@{-} ?<*\dir{<},
(71,4);(63,12) **@{-}, (71,8);(67,4) **@{-}, (63,12);(59,16) **@{-},
(63,16);(59,12) **@{-}, (73,0);(71,4) **\crv{(73,2)}, (73,12);(71,8)
**\crv{(73,10)}, (65,0);(67,4) **\crv{(65,2)}, (63,16);(65,20)
**\crv{(65,18)}, (59,16);(57,20) **\crv{(57,18)}, (59,12);(57,8)
**\crv{(57,10)}, (73,12);(73,20) **@{-}, (57,8);(57,0) **@{-},
(61,12);(61,16) **@{-}, 
(61.1,12);(61.1,16) **@{-},
(60.9,12);(60.9,16) **@{-}, 
(61.2,12);(61.2,16) **@{-},
(60.8,12);(60.8,16) **@{-},
(57,5) *{\wedge},(73,15) *{\wedge},(65,10) *{\lrcorner},(58,17) *{\ulcorner}, (71.5,3) *{\ulcorner},(66.3,3) *{\llcorner},(63.7,16.8) *{\llcorner},
(67,6);(71,6) **@{-}, 
(67,6.1);(71,6.1) **@{-}, 
(67,5.9);(71,5.9) **@{-}, 
(67,6.2);(71,6.2) **@{-},
(67,5.8);(71,5.8) **@{-},  
(-5,10)*{\Gamma_7:}, 
 \endxy}

\vskip.3cm

\centerline{ 
\xy (7,20);(14.2,11) **@{-}, (15.8,9);(17.4,7) **@{-},
(19,5);(23,0) **@{-}, (13,20);(7,12) **@{-}, (7,12);(11.2,7) **@{-},
(12.7,5.2);(14.4,3.2) **@{-}, (15.7,1.6);(17,0) **@{-},
(17,20);(23,12) **@{-}, (13,0);(23,12) **@{-}, (7,0);(23,20) **@{-},
(10,18);(10,14) **@{-}, (10.1,18);(10.1,14) **@{-},
(9.9,18);(9.9,14) **@{-}, (10.2,18);(10.2,14) **@{-},
(9.8,18);(9.8,14) **@{-}, (18,16);(22,16) **@{-},
(18,16.1);(22,16.1) **@{-}, (18,15.9);(22,15.9) **@{-},
(18,16.2);(22,16.2) **@{-}, (18,15.8);(22,15.8) **@{-},
 (8.5,17.7) *{\ulcorner}, (18.5,17.7) *{\ulcorner}, (9.1,9) *{\ulcorner}, (12,18.4) *{\llcorner}, (14.5,1.6) *{\llcorner}, (21.8,18.4) *{\llcorner}, (17,12) *{\urcorner}, (21.7,1.6) *{\lrcorner},
(35,11);(45,11) **@{-} ?>*\dir{>}, (35,9);(45,9) **@{-} ?<*\dir{<},
(73,20);(65.8,11) **@{-}, (64.2,9);(62.6,7) **@{-}, (61,5);(57,0)
**@{-}, (67,20);(73,12) **@{-}, (73,12);(68.8,7) **@{-},
(67.3,5.2);(65.6,3.2) **@{-}, (64.3,1.6);(63,0) **@{-},
(63,20);(57,12) **@{-}, (67,0);(57,12) **@{-}, (73,0);(57,20)
**@{-},
 (58.5,17.7) *{\ulcorner}, (68.5,17.7) *{\ulcorner}, (59.1,9) *{\ulcorner}, (62,18.4) *{\llcorner}, (64,1.2) *{\llcorner}, (71.8,18.4) *{\llcorner}, (67,12) *{\urcorner}, (71.7,1.6) *{\lrcorner},
(60,18);(60,14) **@{-}, (60.1,18);(60.1,14) **@{-},
(59.9,18);(59.9,14) **@{-}, (60.2,18);(60.2,14) **@{-},
(59.8,18);(59.8,14) **@{-}, (68,16);(72,16) **@{-},
(68,16.1);(72,16.1) **@{-}, (68,15.9);(72,15.9) **@{-},
(68,16.2);(72,16.2) **@{-}, (68,15.8);(72,15.8) **@{-},
(-5,10)*{\Gamma_{8}:}, 
\endxy}
\caption{A generating set of oriented Yoshikawa moves}
\label{fig-moves-type-II-o}
\end{center}
\end{figure} 

On many occasions it is necessary to minimize the number of Yoshikawa moves on marked graph diagrams when one checks that a certain function from marked graph diagrams defines a surface-link invariant. 
A collection $S$ of unoriented (oriented, resp.) Yoshikawa moves is called a {\it generating set} of unoriented (oriented, resp.) Yoshikawa moves if any unoriented (oriented, resp.) Yoshikawa move is obtained by a finite sequence of plane isotopies and the moves in the set $S$. 

 The purpose of this paper is to provide some generating sets of Yoshikawa moves on marked graph diagrams representing unoriented surface-links, and also oriented surface-links. The Main Theorems are the following:

\begin{theorem}\label{main-thm-1-1}
Let 
$\mathfrak S=\{\Omega_1, \Omega_2, \Omega_3, \Omega_{4}, \Omega'_{4}, \Omega_{5}, \Omega_6, \Omega'_6, \Omega_7, \Omega_8\},$
the set of the moves illustrated in Fig.~\ref{fig-moves-type-II}. This is a generating set of unoriented Yoshikawa moves.
\end{theorem}

\begin{theorem}\label{main-thm-1-2}
Let
$\mathfrak S_1=\{\Gamma_1, \Gamma'_1, \Gamma_2, \Gamma_3,\Gamma_4, \Gamma'_4, \Gamma_5, \Gamma_6, \Gamma'_6, \Gamma_7, \Gamma_8\},$
the set of the moves illustrated in Fig.~\ref{fig-moves-type-II-o}. This is a generating set of oriented Yoshikawa moves.
\end{theorem}

\begin{theorem}\label{main-thm-1-3}
Let 
$\mathfrak S_2=\{\Gamma_1, \Gamma_{1a}, \Gamma_{2b}, \Gamma_{2c}, \Gamma_{3a}, \Gamma_4, \Gamma'_4, \Gamma_5, \Gamma_6, \Gamma'_6, \Gamma_7, \Gamma_8\},$
where $\Gamma_{1a},$ $\Gamma_{2b},$ $\Gamma_{2c}$, and $\Gamma_{3a}$ are shown in Fig.~\ref{fig-r123o}. This is a generating set of oriented Yoshikawa moves.
\end{theorem}

We remark that there are two types of oriented Reidemeister moves $\Omega_3$ and $\Omega_{3a}$ of type III. One is a {\it cyclic} $\Omega_3$ (or $\Omega_{3a}$) move of which the central triangle have a cyclic orientation. The other is a {\it non-cyclic} $\Omega_3$ (or $\Omega_{3a}$) move of which the central triangle does not have a cyclic orientation. Theorem \ref{main-thm-1-2} gives a generating set $\mathfrak S_1$ of oriented Yoshikawa moves that contains a cyclic 
$\Omega_3$ move $\Gamma_3$. While Theorem \ref{main-thm-1-3} gives a generating set $\mathfrak S_2$ of oriented Yoshikawa moves that contains a non-cyclic 
$\Omega_{3a}$ move $\Gamma_{3a}$ with all three positive crossings. 

The rest of this paper is organized as follows. In Section \ref{sect-mg-crm}, we review minimal generating set of classical Reidemeister moves. In Section \ref{sect-omgr-osl}, we recall the marked graph representation of surface-links. In Section \ref{sect-gs-uoym}, we prove Theorem \ref{main-thm-1-1}. In Section \ref{sect-gs-oym}, we prove Theorem \ref{main-thm-1-2} and Theorem \ref{main-thm-1-3}. In Section \ref{sect-ind-ym}, we discuss independence of certain Yoshikawa moves from the other moves.


\section{Minimal generating sets of Reidemeister moves}
\label{sect-mg-crm}

In this section we recall minimal generating sets of classical Reidemeister moves from \cite{Po}. K. Reidemeister introduced three types of local moves on link diagrams and showed that any two link diagrams representing the same link are transformed into each other by a finite sequence of plane isotopies and the local moves as shown in Fig.~\ref{fig-rmove} (cf. \cite{Re}), nowadays called Reidemeister moves.

\begin{figure}[ht]
\centerline{
\xy (0,0);(4,7) **\crv{(2,7)}, (4,3);(4,7)
**\crv{(6,3)&(6,7)}, (0,10);(1.5,6) **\crv{(0.2,6.5)},
(2.5,4);(4,3) **\crv{(3,3.1)},
\endxy
\xy (2,5);(6,5) **@{-} ?>*\dir{>} ?<*\dir{<},
(3,-3) *{\Omega_1},(0,0)*{},
\endxy
\hskip 0.2cm
\xy (0,0);(0,10) **\crv{(2,5)},
\endxy
\xy (2,5);(6,5) **@{-} ?>*\dir{>} ?<*\dir{<},
(5,-3) *{\Omega_{1a}},(0,0)*{},
\endxy
\hskip 0.2cm
\xy (0,10);(4,3) **\crv{(2,3)}, (4,7);(4,3)
**\crv{(6,7)&(6,3)}, (0,0);(1.5,4) **\crv{(0.2,3.5)},
(2.5,6);(4,7) **\crv{(3,6.9)},
\endxy
\quad
\xy (0,0);(0,10) **\crv{(9,5)}, (6,10);(3.8,8.5) **\crv{(5,9.5)},
(6,0);(3.8,1.5) **\crv{(5,0.5)}, (2.3,2.8);(2.3,7.2)
**\crv{(0.5,5)},
\endxy
\xy (2,5);(6,5) **@{-} ?>*\dir{>} ?<*\dir{<},
(4,-3) *{\Omega_2},(0,0)*{},
\endxy
\xy (0,0);(0,10) **\crv{(4,5)}, 
(6,0);(6,10) **\crv{(2,5)},
\endxy
\quad
\xy (0,5);(10,5) **\crv{(5,0)}, (7.5,4);(2,10) **\crv{(6.7,8)},
(2,0);(2.2,2.2) **\crv{(2,1)}, (8,0);(7.8,2.2) **\crv{(8,1)},
(2.5,4);(4.4,7.7) **\crv{(3,6.5)}, (5.8,8.8);(8,10) **\crv{(6,9)},
\endxy
\xy (2,5);(6,5) **@{-} ?>*\dir{>} ?<*\dir{<},
(4,-3) *{\Omega_3},(0,0)*{},
\endxy
\hskip 0.2cm
\xy (0,5);(10,5) **\crv{(5,10)}, (2.5,6);(8,0) **\crv{(3.3,2)},
(2,10);(2.2,7.8) **\crv{(2,9)}, (8,10);(7.8,7.8) **\crv{(8,9)},
(7.5,6);(5.6,2.3) **\crv{(7,3.5)}, (4.2,1.2);(2,0) **\crv{(4,1)},
\endxy
 \quad
\xy (0,5);(10,5) **\crv{(5,0)}, (2.5,4);(8,10) **\crv{(3.3,8)},
(8,0);(7.8,2.2) **\crv{(8,1)}, (2,0);(2.2,2.2) **\crv{(2,1)},
(7.5,4);(5.6,7.7) **\crv{(7,6.5)}, (4.2,8.8);(2,10) **\crv{(4,9)},
\endxy
\xy (2,5);(6,5) **@{-} ?>*\dir{>} ?<*\dir{<},
(4,-3) *{\Omega_{3a}},(0,0)*{},
\endxy
\hskip 0.2cm
\xy (0,5);(10,5) **\crv{(5,10)}, (7.5,6);(2,0) **\crv{(6.7,2)},
(8,10);(7.8,7.8) **\crv{(8,9)}, (2,10);(2.2,7.8) **\crv{(2,9)},
(2.5,6);(4.4,2.3) **\crv{(3,3.5)}, (5.8,1.2);(8,0) **\crv{(6,1)},
\endxy
}\caption{Unoriented Reidemeister moves}
\label{fig-rmove}
\end{figure}
\begin{figure}[ht]
\begin{center}
\resizebox{0.80\textwidth}{!}{%
  \includegraphics{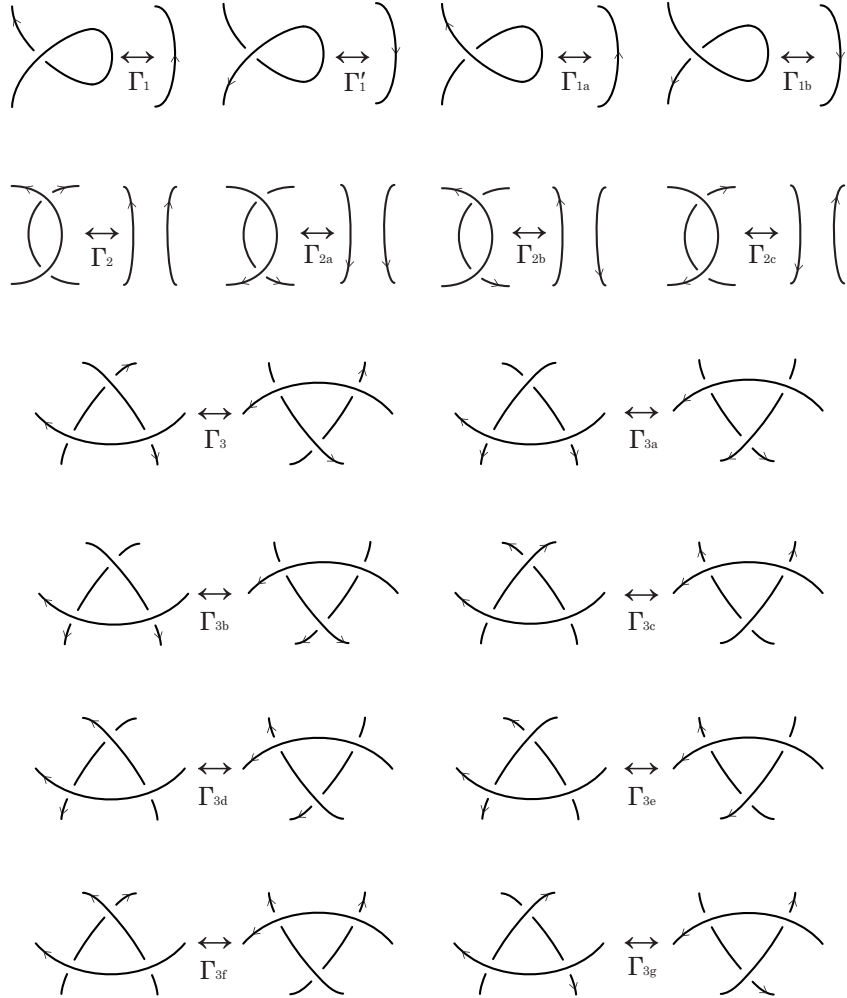} }
\caption{Oriented Reidemeister moves}
\label{fig-r123o}
\end{center}
\end{figure}

A collection $S$ of (oriented) Reidemeister moves is called a {\it generating set} if any (oriented) Reidemeister move is obtained by a finite sequence of plane isotopies and the moves in the collection $S$. A generating set $S$ is said to be {\it minimal} if $S$ does not properly contain any generating set.

We remind the reader that the Reidemeister moves $\Omega_{1}, \Omega_2, \Omega_{3}$ in Fig.~\ref{fig-rmove} is a minimal generating set of unoriented Reidemeister moves, that is, all other Reidemeister moves $\Omega_{1a}$ and $\Omega_{3a}$ in Fig.~\ref{fig-rmove} are obtained by a finite sequence of plane isotopies and the moves $\Omega_{1}, \Omega_2, \Omega_{3}$. We should remark that there are other minimal generating sets of unoriented Reidemeister moves. For example, the Reidemeister moves $\Omega_1, \Omega_2, \Omega_{3a}$ in Fig.~\ref{fig-rmove} is also a minimal generating set of unoriented Reidemeister moves, which seems to be considered in many literatures, for example \cite{AB,Ka2,Re}. 

To deal with oriented link diagrams, depending on orientations of arcs involved in the Reidemeister moves, one may distinguish four different versions of each of the move $\Omega_1$, $\Omega_2$, $\Omega_3$ and $\Omega_{3a}$. For the future use, we here list the oriented Reidemeister moves and name all of them as shown in Fig.~\ref{fig-r123o}. 

It is obvious that any generating set of oriented Reidemeister moves must contain at least one oriented move of each of the moves $\Omega_2$ and $\Omega_3$. In \cite{Ol}, Olof-Petter \"Ostlund showed that any generating set of oriented Reidemeister moves has to contain at least two oriented moves of the move $\Omega_1$. Hence any generating set of oriented Reidemeister moves has to contain at least four moves.
In \cite{Po}, M. Polyak introduced minimal generating sets of four oriented Reidemeister moves, which include two oriented $\Omega_1$ moves, one or two oriented $\Omega_2$ move and one oriented $\Omega_3$ move (either a cyclic $\Omega_3$  move or a non-cyclic $\Omega_3$ move): 

\begin{theorem}\cite[Theorem 1.1]{Po}\label{thm-polyak}
Let $D$ and $D'$ be two oriented link diagrams in $\mathbb R^2$ representing the same oriented link. Then one may pass from $D$ to $D'$ by a finite sequence of plane isotopies and four oriented Reidemeister moves $\Gamma_1, \Gamma'_1, \Gamma_2$ and $\Gamma_3$ in Fig.~\ref{fig-r123o}. 
\end{theorem}

\begin{theorem}\cite[Theorem 1.2]{Po}\label{thm-polyak-1}
Let $S$ be a set of at most five oriented Reidemeister moves containing the move $\Gamma_{3a}$. Then the set $S$ generates all oriented Reidemeister moves if and only if $S$ contains $\Gamma_{2b}$, $\Gamma_{2c}$ and contains one of the pairs $(\Gamma_{1}, \Gamma'_{1})$, $(\Gamma_{1}, \Gamma_{1a})$, 
$(\Gamma_{1b}, \Gamma'_{1})$, or $(\Gamma_{1a}, \Gamma_{1b})$. 
\end{theorem}

\begin{remark}
(1) There are other generating sets for Reidemeister moves. In fact, the set of four oriented Reidemeister moves $\Gamma_1, \Gamma'_1, \Gamma_{2a}$ and $\Gamma_3$ is also a minimal generating set of oriented Reidemeister moves.

(2) Theorem \ref{thm-polyak-1} implies that any generating set of oriented Reidemeister moves which includes the move $\Gamma_{3a}$ contains at least five moves. 
\end{remark}

In the rest of the paper, we carefully deal with Yoshikawa moves on marked graph diagrams of surface-links in $\mathbb R^4$ and provide generating sets of unoriented/oriented Yoshikawa moves, which are analogous to those of unoriented/oriented Reidemeister moves on diagrams of classical links in $\mathbb R^3$.


\section{Marked graph representation of surface-links}\label{sect-omgr-osl}

In this section, we review marked graph diagrams representing surface-links. A {\it marked graph} is a spatial graph $G$ in $\mathbb R^3$ which satisfies the following:
\begin{itemize}
  \item [(1)] $G$ is a finite regular graph with $4$-valent vertices, say $v_1, v_2, . . . , v_n$.
  \item [(2)] Each $v_i$ is a rigid vertex; that is, we fix a rectangular neighborhood $N_i$ homeomorphic to $\{(x, y)|-1 \leq x, y \leq 1\},$ where $v_i$ corresponds to the origin and the edges incident to $v_i$ are represented by $x^2 = y^2$.
  \item [(3)] Each $v_i$ has a {\it marker}, which is the interval on $N_i$ given by  $\{(x, 0)|-1 \leq x \leq 1\}$.
\end{itemize}

An {\it orientation} of a marked graph $G$ is a choice of an orientation for each edge of $G$ in such a way that every vertex in $G$ looks like 
\xy (-5,5);(5,-5) **@{-}, 
(5,5);(-5,-5) **@{-}, 
(3,3.2)*{\llcorner}, 
(-3,-3.4)*{\urcorner}, 
(-2.5,2)*{\ulcorner},
(2.5,-2.4)*{\lrcorner}, 
(3,-0.2);(-3,-0.2) **@{-},
(3,0);(-3,0) **@{-}, 
(3,0.2);(-3,0.2) **@{-}, 
\endxy.
A marked graph $G$ is said to be 
{\it orientable} if it admits an orientation. Otherwise, it is said to be {\it non-orientable}. By an {\it oriented marked graph} we mean an orientable marked graph with a fixed orientation (see Fig.~\ref{fig-ori-mg}). Two oriented marked graphs are said to be {\it equivalent} if they are ambient isotopic in $\mathbb R^3$ with keeping the rectangular neighborhood, marker and orientation. As usual, a marked graph can be described by a link diagram on $\mathbb R^2$ with some $4$-valent vertices equipped with markers.

For $t \in \mathbb R,$ we denote by $\mathbb R^3_t$ the hyperplane of $\mathbb R^4$ whose fourth coordinate
is equal to $t \in \mathbb R$, i.e., $\mathbb R^3_t := \{(x_1, x_2, x_3, x_4) \in
\mathbb R^4~|~ x_4 = t \}$. A surface-link $\mathcal L \subset \mathbb R^4=\mathbb R^3 \times \mathbb R$ can be described in terms of its {\it cross-sections} $\mathcal L_t=\mathcal L \cap \mathbb R^3_t, ~ t \in \mathbb R$ (cf. \cite{Fox}). Let $p:\mathbb R^4 \to \mathbb R$ be the projection given by $p(x_1, x_2, x_3, x_4)=x_4$. Note that any surface-link can be perturbed to a surface-link $\mathcal L$ such that the projection $p : \mathcal L \to \mathbb R$ is a Morse function with finitely many distinct non-degenerate critical values. More especially, it is well known \cite{Ka2,Kaw,KSS,Lo} that any surface-link $\mathcal L$ can be deformed into a surface-link $\mathcal L'$, called a {\it hyperbolic splitting} of $\mathcal L$,
by an ambient isotopy of $\mathbb R^4$ in such a way that
the projection $p: \mathcal L' \to \mathbb R$ satisfies that
all critical points are non-degenerate,
all the index 0 critical points (minimal points) are in $\mathbb R^3_{-1}$,
all the index 1 critical points (saddle points) are in $\mathbb R^3_0$, and
all the index 2 critical points (maximal points) are in $\mathbb R^3_1$.

Let $\mathcal L$ be a surface-link and let ${\mathcal L'}$ be a hyperbolic splitting of $\mathcal L.$ Then the cross-section $\mathcal L'_0=\mathcal L'\cap \mathbb R^3_0$ at $t=0$ is a spatial $4$-valent regular graph in $\mathbb R^3_0$. We give a marker at each $4$-valent vertex (saddle point) that indicates how the saddle point opens up above as illustrated in Fig.~\ref{sleesan2:fig1}. 

\begin{figure}[h]
\begin{center}
\resizebox{0.55\textwidth}{!}{%
  \includegraphics{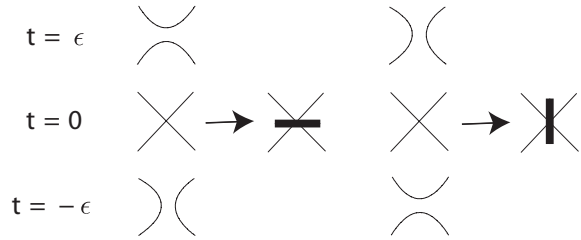} }
\caption{Marking of a vertex} \label{sleesan2:fig1}
\end{center}
\end{figure} 

The resulting marked graph $G$ is called a {\it marked graph} presenting $\mathcal L$. As usual, $G$ is described by a diagram $D$ in $\mathbb R^2$, which is a generic projection on $\mathbb R^2$ of $G$ with over/under crossing information for each double point, as in a classical link diagram, such that the restriction to a small rectangular neighborhood of each marked vertex is a homeomorphism. Such a diagram $D$ is called a {\it marked graph diagram} (or {\it ch-diagram} (cf. \cite{So})) {\it presenting} $\mathcal L$.  

When $\mathcal L$ is an oriented surface-link, we choose an orientation for each edge of $\mathcal L'_0$ that coincides with the induced orientation on the boundary of $\mathcal L' \cap \mathbb R^3 \times (-\infty, 0]$ by the orientation of $\mathcal L'$ inherited from the orientation of $\mathcal L$. The resulting oriented marked graph diagram $D$ is called an {\it oriented marked graph diagram} presenting the oriented surface-link $\mathcal L$. 

\begin{figure}[ht]
\begin{center}
\resizebox{0.50\textwidth}{!}{%
\includegraphics{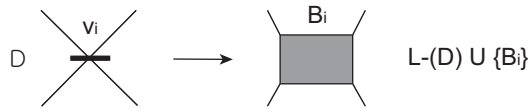} }
\caption{A band attached to $L_-(D)$ at $v_i$}\label{fig-orbd}
\end{center}
\end{figure}

Now let $D$ be a given admissible marked graph diagram with marked vertices $v_1, \ldots, v_n$. We define a surface $F(D) \subset \mathbb R^3 \times [-1,1]$ by
\begin{equation*}
(\mathbb R^3_t, F(D) \cap \mathbb R^3_t)=\left\{%
\begin{array}{ll}
    (\mathbb R^3, L_+(D)) & \hbox{for $0 < t \leq 1$,} \\
    \bigg( \mathbb R^3, L_-(D) \cup \bigg( \overset{n}{\underset{i=1}{\cup}} B_i \bigg) \bigg) & \hbox{for $t = 0$,} \\
    (\mathbb R^3, L_-(D)) & \hbox{for $-1 \leq t < 0$,} \\
\end{array}%
\right.
\end{equation*}
where $B_i (i=1, \ldots,n)$ is a band attached to $L_-(D)$ at the marked vertex $v_i$ as shown in Fig.~\ref{fig-orbd}.
We call $F(D)$ the {\it proper surface associated with} $D$.
Since $D$ is admissible, we can obtain a surface-link from $F(D)$ by attaching trivial disks in $\mathbb R^3 \times [1, \infty)$ and another trivial disks in $\mathbb R^3 \times (-\infty, 1]$.  We denote this surface-link by $\mathcal S(D)$, and call it
the {\it surface-link associated with $D$}. It is known that the isotopy type of $\mathcal S(D)$ does not depend on choices of trivial disks (cf. \cite{Ka2,KSS}). It is straightforward from the construction of $\mathcal S(D)$ that $D$ is a marked graph diagram presenting $\mathcal S(D)$. 

It is known that $D$ is orientable if and only if $F(D)$ is an orientable surface.  When $D$ is oriented, the resolutions
$L_-(D)$ and $L_+(D)$ have orientations induced from the orientation of $D$, and we assume $F(D)$ is oriented so that the induced orientation on $L_+(D) = \partial F(D) \cap \mathbb R^3_1$ matches the orientation of $L_+(D)$.

 Theorem \ref{thm-eqiv-ch-digs} in Section 1 shows that a surface-link is completely described by its marked graph diagram modulo unoriented Yoshikawa's moves, and also shows that an oriented surface-link is completely described by its oriented marked graph diagram modulo oriented Yoshikawa's moves.


\section{Proof of Theorem \ref{main-thm-1-1}}
\label{sect-gs-uoym}

To prove Theorem \ref{main-thm-1-1}, it suffices to show that the mirror moves in Fig.~\ref{fig-r45678n} and the Reidemeister moves $\Omega_{1a}$ and $\Omega_{3a}$ in Fig.~\ref{fig-rmove} are generated by the $10$ moves in the set $\mathfrak S$ in Theorem \ref{main-thm-1-1}. This is going to be done in several steps. We begin with the following 

\begin{lemma}\label{main-thm-3-2-lem1}
The move $\Omega_{4a}$ is realized by $\Omega_{2}$, $\Omega_{4}$ and plane isotopies. The move $\Omega'_{4a}$ is realized by $\Omega_{2}$, $\Omega'_{4}$ and plane isotopies.
\end{lemma}

\begin{proof} 
See Fig.~\ref{fig-m4}.
\end{proof}

\begin{figure}[ht]
\begin{center}
\resizebox{0.80\textwidth}{!}{%
  \includegraphics{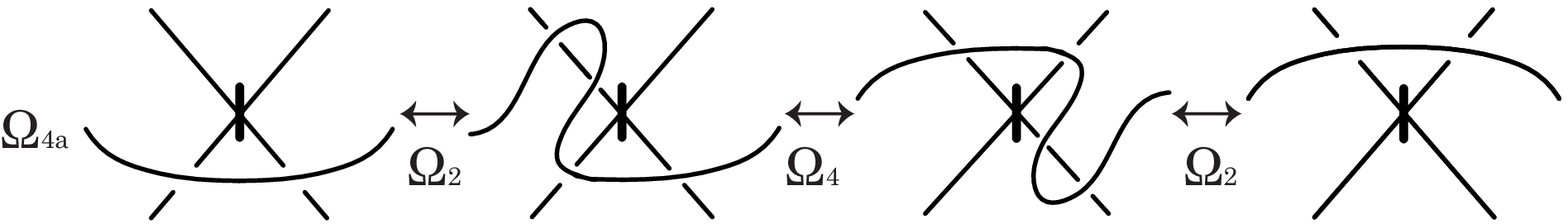}}
\end{center}
\begin{center}
\resizebox{0.80\textwidth}{!}{%
  \includegraphics{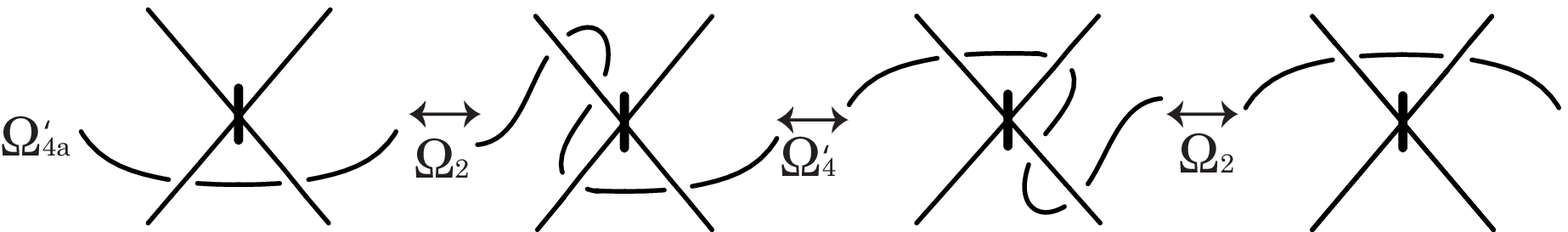} }
\caption{Mirror moves of $\Omega_4$ and $\Omega'_4$}\label{fig-m4}
\end{center}
\begin{center}
\resizebox{0.75\textwidth}{!}{%
  \includegraphics{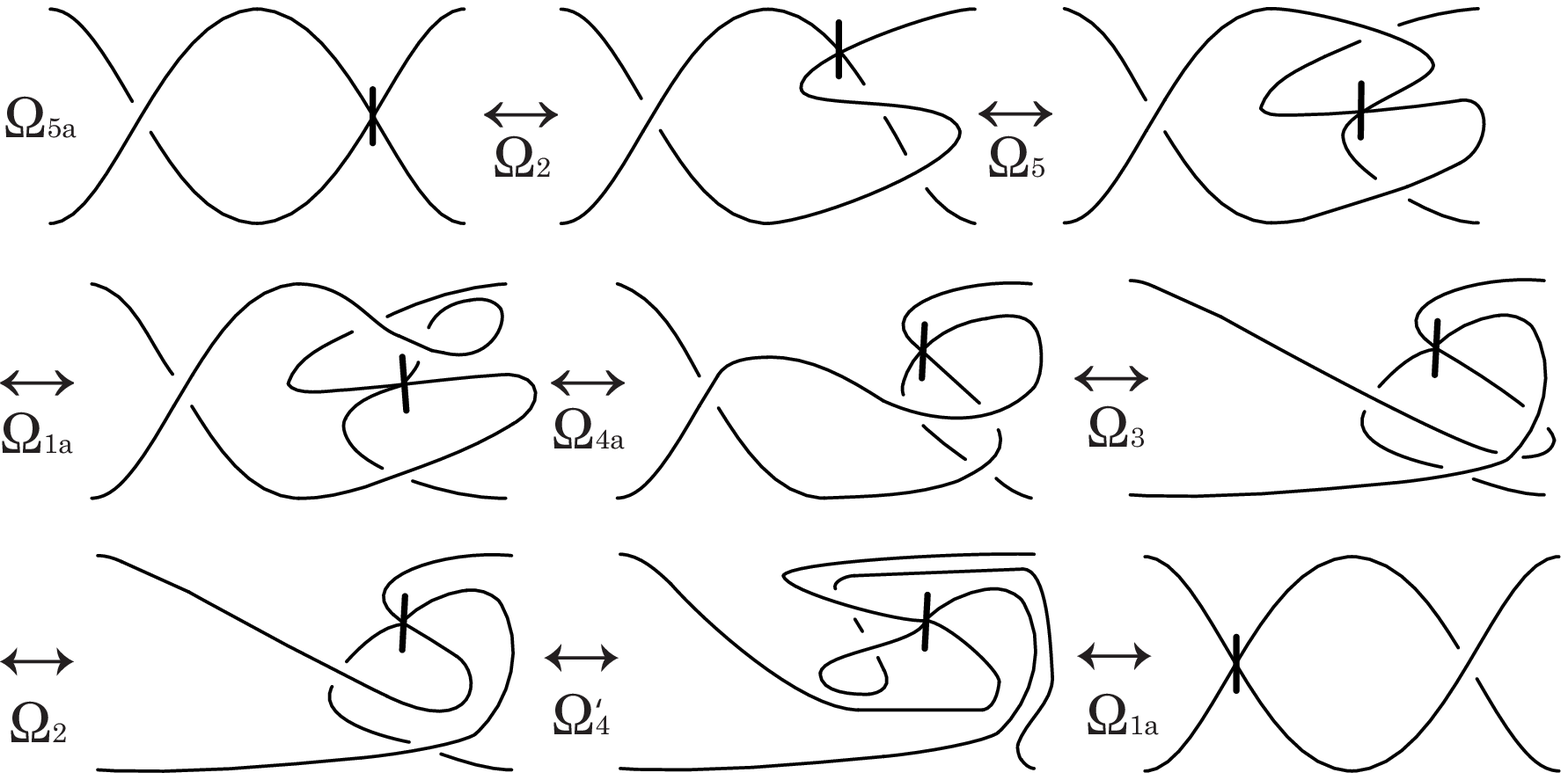}}
\end{center}
\begin{center}
\resizebox{0.90\textwidth}{!}{%
  \includegraphics{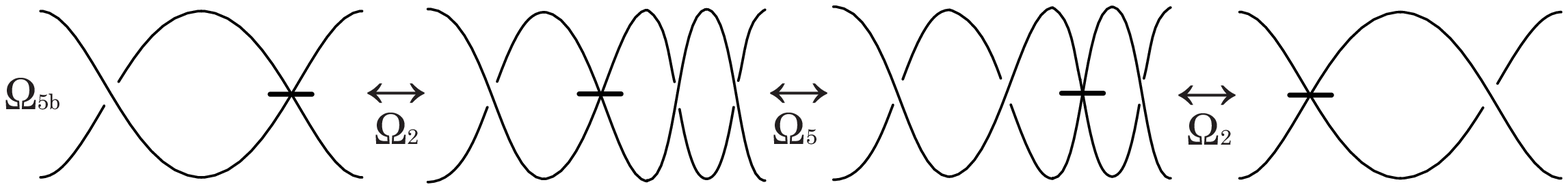} }
\end{center}
\begin{center}
\resizebox{0.95\textwidth}{!}{%
  \includegraphics{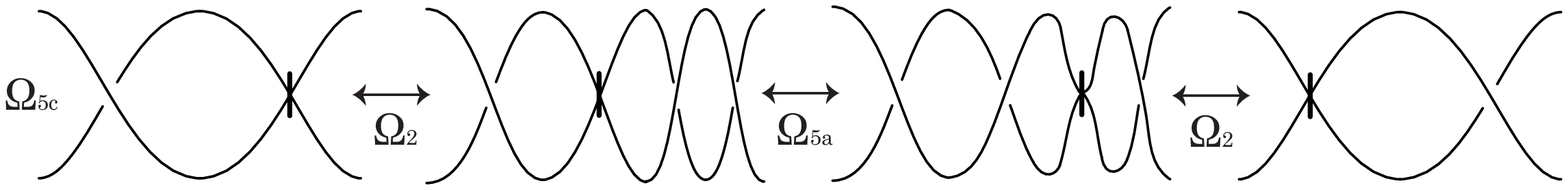}}
  \caption{Mirror moves of $\Omega_5$}\label{fig-m5}
\end{center}
\end{figure}
\begin{figure}[ht]
\begin{center}
\resizebox{0.70\textwidth}{!}{%
\includegraphics{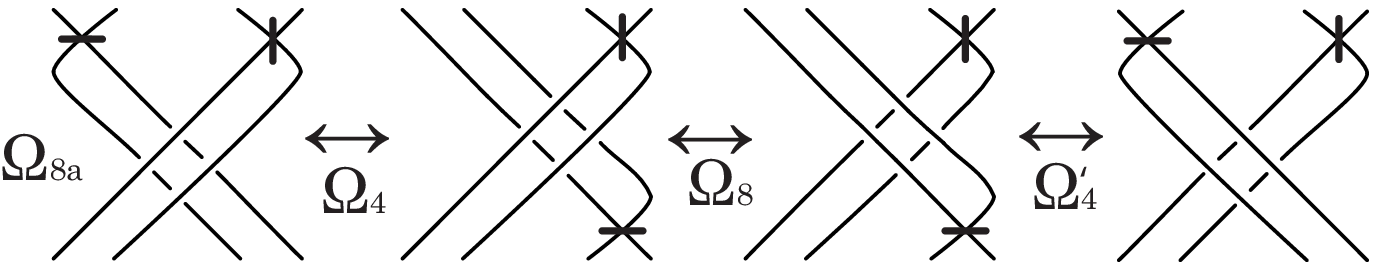}}
\caption{Mirror move of $\Omega_8$}\label{fig-m8}
\end{center}
\end{figure}

\begin{lemma}\label{main-thm-3-2-lem2}
The move $\Omega_{5a}$ is realized by $\Omega_{1a}$, $\Omega_{2}$, $\Omega_{3}$, $\Omega'_{4}$, $\Omega_{4a}$, $\Omega_{5}$ and plane isotopies. The move $\Omega_{5b}$ is realized by $\Omega_{2}$, $\Omega_{5}$ and plane isotopies. The move $\Omega_{5c}$ is realized by $\Omega_{2}$, $\Omega_{5a}$ and plane isotopies.
\end{lemma}

\begin{proof} 
See Fig.~\ref{fig-m5}.
\end{proof}

\begin{lemma}\label{main-thm-3-2-lem5}
The move $\Omega_{8a}$ is realized by $\Omega_{4}$, $\Omega'_{4}$, $\Omega_8$ and plane isotopies.
\end{lemma}

\begin{proof}
See Fig.~\ref{fig-m8}.
\end{proof}

{\bf Proof of Theorem \ref{main-thm-1-1}.} 
Let $M$ be a move of the unoriented Yoshikawa moves.
Note that the three Reidemeister moves $\Omega_1, \Omega_2, \Omega_3$ generate all unoriented Reidemeister moves and so do $\Omega_{1a}$ and $\Omega_{3a}$ (cf. Section 2). From Lemmas \ref{main-thm-3-2-lem1}, \ref{main-thm-3-2-lem2} and \ref{main-thm-3-2-lem5}, it is seen that $M$ lies in the set $\mathfrak S$ in Theorem \ref{main-thm-1-1}, or $M$ can be generated by the moves lie in $\mathfrak S$. This completes the proof of Theorem \ref{main-thm-1-1}.\hfill$\Box$

\begin{remark}
There are minimal generating sets of classical Reidemeister moves other than $\{\Omega_1, \Omega_2, \Omega_3\}$. For example, $\{\Omega_{1a}, \Omega_2, \Omega_{3a}\}$. Indeed, replacing any generating set of Reidemeister moves with the moves $\Omega_1, \Omega_2, \Omega_3$ yields a generating set of unoriented Yoshikawa moves. This fact and Lemmas \ref{main-thm-3-2-lem1}-\ref{main-thm-3-2-lem5} enable us to do that we can choose a generating set of unoriented Yoshikawa moves other than one given in Theorem \ref{main-thm-1-1}. For example, it is easily seen that $\{\Omega_{1a}, \Omega_2, \Omega_{3a}, \Omega_{4}, \Omega'_{4}, \Omega_{5}, \Omega_{6},$ $\Omega'_{6},$ $\Omega_{7}, \Omega_{8}\}$  and 
$\{\Omega_{1a}, \Omega_2, \Omega_{3a}, \Omega_{4a}, \Omega_{4b}, \Omega_{5}, \Omega_{6}, \Omega'_{6}, \Omega_{7}, \Omega_{8}\}$ are also generating sets of unoriented Yoshikawa moves.
\end{remark}


\section{Proof of Theorem \ref{main-thm-1-2} and Theorem \ref{main-thm-1-3}}\label{sect-gs-oym}

We remind that the oriented Yoshikawa moves are the moves $\Omega_i(i = 1, 2, \ldots, 8),$ $\Omega'_4,$ $\Omega'_6$ in Fig.~\ref{fig-moves-type-II}, $\Omega_{4a}, \Omega'_{4a}, \Omega_{5a},$ $\Omega_{5b},$ $\Omega_{5c}, \Omega_{8a}$ in Fig.~\ref{fig-r45678n} and Fig.~\ref{fig-rmove} with all possible orientations. On the other hand, by virtue of Theorem \ref{main-thm-1-1}, we see that any such an oriented Yoshikawa move is generated by plane isotopies and some moves lie in the generating set $\mathfrak S=\{\Omega_{1}, \Omega_2, \Omega_{3}, \Omega_{4}, \Omega'_{4}, \Omega_{5}, \Omega_{6}, \Omega'_{6}, \Omega_{7}, \Omega_{8}\}$ of 10 Yoshikawa moves with all possible orientations.
In Section \ref{sect-mg-crm}, we have discussed the oriented Reidemeister moves $\Gamma_{1}, \Gamma_2, \Gamma_{3}$ (see Fig.~\ref{fig-r123o}). Now it is easily seen that all possible Yoshikawa moves $\Omega_{4}, \ldots, \Omega_8, \Omega'_4, \Omega'_6$ with orientations are the oriented moves $\Gamma_{4}, \ldots, \Gamma_8, \Gamma'_4, \Gamma'_6$ in Fig.~\ref{fig-moves-type-II-o} and the moves $\Gamma_{4a}, \Gamma'_{4a}, \Gamma_{6a}, \Gamma'_{6a}, \Gamma_{7a}, \Gamma_{8a}, \Gamma_{8b}, \Gamma_{8c}$ shown in Fig.~\ref{fig-om2}. We note that the move $\Gamma_5$ with reverse orientation is the move $\Gamma_5$ itself. 

\begin{figure}[ht]
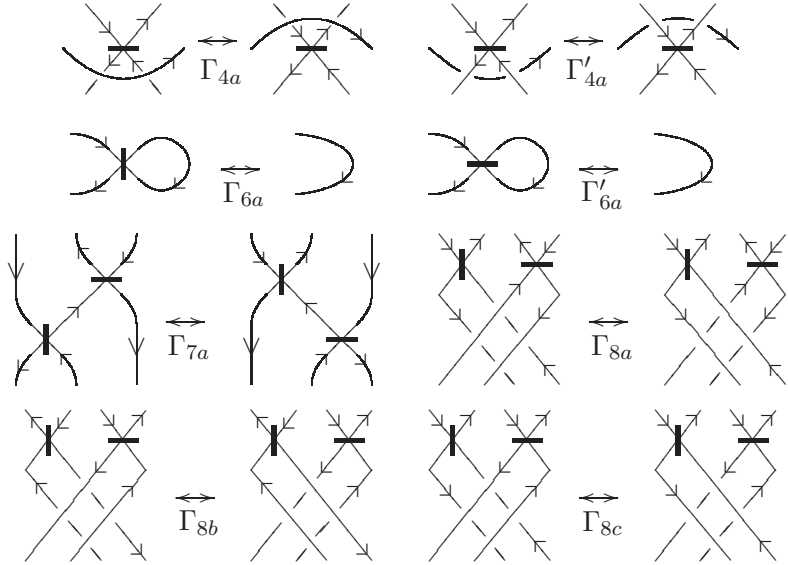

\centerline{
\xy
(7,6);(23,6) **\crv{(15,-2)}, 
(10,0);(11.5,1.8) **@{-},
(12.5,3);(20,12) **@{-}, 
(10,12);(17.5,3) **@{-}, 
(18.5,1.8);(20,0) **@{-}, 
(13,6);(17,6) **@{-}, (13,6.1);(17,6.1) **@{-}, (13,5.9);(17,5.9)
**@{-}, (13,6.2);(17,6.2) **@{-}, (13,5.8);(17,5.8) **@{-},
  (13.8,4.5) *{\llcorner}, (17.5,8.6) *{\urcorner}, (17,3.1) *{\ulcorner},  (21,3.5) *{\urcorner},(12.5,9) *{\lrcorner},
(25,7);(30,7) **@{-} ?>*\dir{>} ?<*\dir{<}, (32,7) *{},
(28,3)*{\Gamma_{4a}},
\endxy
\xy (57,6);(73,6)  **\crv{(65,14)}, 
(70,12);(68.5,10.2) **@{-},
(67.5,9);(60,0) **@{-}, 
(70,0);(62.5,9) **@{-}, 
(61.5,10.2);(60,12) **@{-}, 
(63,6);(67,6) **@{-}, (63,6.1);(67,6.1) **@{-}, (63,5.9);(67,5.9)
**@{-}, (63,6.2);(67,6.2) **@{-}, (63,5.8);(67,5.8) **@{-},
(63,3.4) *{\llcorner}, (66.4,7.3) *{\urcorner}, (68,2) *{\ulcorner}, (70.9,8) *{\lrcorner},(63,8.5) *{\lrcorner},
\endxy\quad\quad\xy 
  (13,2.2);(17,2.2)  **\crv{(15,1.7)}, 
  (7,6);(11,3)  **\crv{(10,3.5)}, 
  (23,6);(19,3)  **\crv{(20,3.5)}, 
 (10,0);(20,12) **@{-}, 
(10,12);(20,0) **@{-}, 
(13,6);(17,6) **@{-}, (13,6.1);(17,6.1) **@{-}, (13,5.9);(17,5.9)
**@{-}, (13,6.2);(17,6.2) **@{-}, (13,5.8);(17,5.8) **@{-}, 
 (13.8,4.5) *{\llcorner}, (17.5,8.6) *{\urcorner}, (17,3.1) *{\ulcorner},  (21,3.5) *{\urcorner},(12.5,9) *{\lrcorner},
(25,7);(30,7) **@{-} ?>*\dir{>} ?<*\dir{<}, (32,7) *{},
(28,3)*{\Gamma'_{4a}},
\endxy
\xy
   (63,9.8);(67,9.8)  **\crv{(65,10.3)}, 
  (57,6);(61,9)  **\crv{(60,8.5)}, 
  (73,6);(69,9)  **\crv{(70,8.5)}, 
(70,12);(60,0) **@{-}, 
(70,0);(60,12) **@{-}, 
(63,6);(67,6) **@{-}, (63,6.1);(67,6.1) **@{-}, (63,5.9);(67,5.9)
**@{-},(63,6.2);(67,6.2) **@{-}, (63,5.8);(67,5.8) **@{-}, 
(63,3.4) *{\llcorner}, (66.4,7.3) *{\urcorner}, (68,2) *{\ulcorner}, (70.9,8) *{\lrcorner},(63,8.5) *{\lrcorner},
\endxy}

\vskip 0.1cm

\centerline{
\xy (12,6);(16,2) **@{-}, (12,2);(16,6) **@{-},
(16,6);(22,6) **\crv{(18,8)&(20,8)}, (16,2);(22,2)
**\crv{(18,0)&(20,0)}, (22,6);(22,2) **\crv{(23.5,4)}, (7,8);(12,6) **\crv{(10,8)}, (7,0);(12,2) **\crv{(10,0)}, (11,1.2) *{\llcorner}, (11,7) *{\lrcorner},(21.5,1.5)*{\llcorner},
(14,6);(14,2) **@{-}, (14.1,6);(14.1,2) **@{-}, (13.9,6);(13.9,2)
**@{-}, (14.2,6);(14.2,2) **@{-}, (13.8,6);(13.8,2) **@{-},
(27,3.3);(32,3.3) **@{-} ?>*\dir{>} ?<*\dir{<}, 
(37,12) *{}, (30,0)*{\Gamma_{6a}},
 \endxy
\xy
(57,8);(57,0) **\crv{(67,7)&(67,1)}, 
(63.8,2.4) *{\llcorner},
\endxy
\quad\quad 
\xy (12,6);(16,2) **@{-}, (12,2);(16,6) **@{-},
(16,6);(22,6) **\crv{(18,8)&(20,8)}, (16,2);(22,2)
**\crv{(18,0)&(20,0)}, (22,6);(22,2) **\crv{(23.5,4)}, (7,8);(12,6) **\crv{(10,8)}, (7,0);(12,2) **\crv{(10,0)}, (11,1.2) *{\llcorner}, (11,7) *{\lrcorner},(21.5,1.5)*{\llcorner},
(12,4);(16,4) **@{-}, (12,4.1);(16,4.1) **@{-},
(12,4.2);(16,4.2) **@{-}, (12,3.9);(16,3.9) **@{-},
(12,3.8);(16,3.8) **@{-},
(27,3.3);(32,3.3) **@{-} ?>*\dir{>} ?<*\dir{<}, 
(37,12) *{}, (30,0)*{\Gamma'_{6a}},
 \endxy
\xy (57,8);(57,0) **\crv{(67,7)&(67,1)}, 
 (63.8,2.4) *{\llcorner},
\endxy}

\vskip.3cm

\centerline{\xy (9,4);(17,12) **@{-}, (9,8);(13,4) **@{-},
(17,12);(21,16) **@{-}, (17,16);(21,12) **@{-}, (7,0);(9,4)
**\crv{(7,2)}, (7,12);(9,8) **\crv{(7,10)}, (15,0);(13,4)
**\crv{(15,2)}, (17,16);(15,20) **\crv{(15,18)}, (21,16);(23,20)
**\crv{(23,18)}, (21,12);(23,8) **\crv{(23,10)}, (7,12);(7,20)
**@{-}, (23,8);(23,0) **@{-},
(11,4);(11,8) **@{-}, 
(10.9,4);(10.9,8) **@{-}, 
(11.1,4);(11.1,8) **@{-}, 
(10.8,4);(10.8,8) **@{-}, 
(11.2,4);(11.2,8) **@{-},
(17,14);(21,14) **@{-}, 
(17,14.1);(21,14.1) **@{-},
(17,13.9);(21,13.9) **@{-}, 
(17,14.2);(21,14.2) **@{-},
(17,13.8);(21,13.8) **@{-},
(7,15) *{\vee},(23,5) *{\vee},(14.6,9) *{\urcorner},(8,2.7) *{\llcorner}, (21.9,17) *{\llcorner},(16,16.8) *{\ulcorner},(13.7,2.8) *{\ulcorner},
(27,8.3);(32,8.3) **@{-} ?>*\dir{>} ?<*\dir{<}, (37,12) *{},(30,5)*{\Gamma_{7a}},
 \endxy
\xy
(71,4);(63,12) **@{-}, (71,8);(67,4) **@{-}, (63,12);(59,16) **@{-},
(63,16);(59,12) **@{-}, (73,0);(71,4) **\crv{(73,2)}, (73,12);(71,8)
**\crv{(73,10)}, (65,0);(67,4) **\crv{(65,2)}, (63,16);(65,20)
**\crv{(65,18)}, (59,16);(57,20) **\crv{(57,18)}, (59,12);(57,8)
**\crv{(57,10)}, (73,12);(73,20) **@{-}, (57,8);(57,0) **@{-},
(61,12);(61,16) **@{-}, 
(61.1,12);(61.1,16) **@{-},
(60.9,12);(60.9,16) **@{-}, 
(61.2,12);(61.2,16) **@{-},
(60.8,12);(60.8,16) **@{-},
(57,5) *{\vee},(73,15) *{\vee},(65,9.5) *{\ulcorner},(58,17.1) *{\lrcorner}, (71.5,3.5) *{\lrcorner},(66.3,2.8) *{\urcorner},(63.7,16.3) *{\urcorner},
(67,6);(71,6) **@{-}, 
(67,6.1);(71,6.1) **@{-}, 
(67,5.9);(71,5.9) **@{-}, 
(67,6.2);(71,6.2) **@{-},
(67,5.8);(71,5.8) **@{-},  
\endxy
\quad\quad
\xy (7,20);(14.2,11) **@{-}, (15.8,9);(17.4,7) **@{-},
(19,5);(23,0) **@{-}, (13,20);(7,12) **@{-}, (7,12);(11.2,7) **@{-},
(12.7,5.2);(14.4,3.2) **@{-}, (15.7,1.6);(17,0) **@{-},
(17,20);(23,12) **@{-}, (13,0);(23,12) **@{-}, (7,0);(23,20) **@{-},
(10,18);(10,14) **@{-}, (10.1,18);(10.1,14) **@{-},
(9.9,18);(9.9,14) **@{-}, (10.2,18);(10.2,14) **@{-},
(9.8,18);(9.8,14) **@{-}, (18,16);(22,16) **@{-},
(18,16.1);(22,16.1) **@{-}, (18,15.9);(22,15.9) **@{-},
(18,16.2);(22,16.2) **@{-}, (18,15.8);(22,15.8) **@{-},
 (8,18.8) *{\lrcorner}, %
 (18.5,17.7) *{\ulcorner}, 
 (9.1,9.5) *{\lrcorner}, %
 (11.8,18.1) *{\urcorner}, %
 (21,9.6) *{\llcorner}, 
 (21.8,18.4) *{\llcorner}, 
 (17,12) *{\urcorner}, 
 (21.7,1.3) *{\ulcorner},%
(27,8.3);(32,8.3) **@{-} ?>*\dir{>} ?<*\dir{<}, (37,12) *{},(30,5)*{\Gamma_{8a}},
 \endxy
\xy
(73,20);(65.8,11) **@{-}, (64.2,9);(62.6,7) **@{-}, (61,5);(57,0)
**@{-}, (67,20);(73,12) **@{-}, (73,12);(68.8,7) **@{-},
(67.3,5.2);(65.6,3.2) **@{-}, (64.3,1.6);(63,0) **@{-},
(63,20);(57,12) **@{-}, (67,0);(57,12) **@{-}, (73,0);(57,20)
**@{-},
 (58,18.8) *{\lrcorner}, %
 (68.5,17.7) *{\ulcorner}, 
 (59.1,9.5) *{\lrcorner}, %
 (61.8,18.1) *{\urcorner}, %
 (71,9.6) *{\llcorner},  
 (71.8,18.4) *{\llcorner}, 
 (67,12) *{\urcorner}, 
 (71.7,1.3) *{\ulcorner},%
(60,18);(60,14) **@{-}, (60.1,18);(60.1,14) **@{-},
(59.9,18);(59.9,14) **@{-}, (60.2,18);(60.2,14) **@{-},
(59.8,18);(59.8,14) **@{-}, (68,16);(72,16) **@{-},
(68,16.1);(72,16.1) **@{-}, (68,15.9);(72,15.9) **@{-},
(68,16.2);(72,16.2) **@{-}, (68,15.8);(72,15.8) **@{-},
\endxy}

\vskip.3cm

\centerline{ \xy (7,20);(14.2,11) **@{-}, (15.8,9);(17.4,7) **@{-},
(19,5);(23,0) **@{-}, (13,20);(7,12) **@{-}, (7,12);(11.2,7) **@{-},
(12.7,5.2);(14.4,3.2) **@{-}, (15.7,1.6);(17,0) **@{-},
(17,20);(23,12) **@{-}, (13,0);(23,12) **@{-}, (7,0);(23,20) **@{-},
(10,18);(10,14) **@{-}, (10.1,18);(10.1,14) **@{-},
(9.9,18);(9.9,14) **@{-}, (10.2,18);(10.2,14) **@{-},
(9.8,18);(9.8,14) **@{-}, (18,16);(22,16) **@{-},
(18,16.1);(22,16.1) **@{-}, (18,15.9);(22,15.9) **@{-},
(18,16.2);(22,16.2) **@{-}, (18,15.8);(22,15.8) **@{-},
 (8.5,17.7) *{\ulcorner},%
 (18,18.8) *{\lrcorner}, 
 (9.1,9) *{\ulcorner}, %
 (12,18.4) *{\llcorner}, %
 (20.5,8.5) *{\urcorner}, 
 (21.8,18) *{\urcorner}, 
 (17,12.5) *{\llcorner}, 
 (21.7,1.6) *{\lrcorner},%
(27,8.3);(32,8.3) **@{-} ?>*\dir{>} ?<*\dir{<}, (37,12) *{},(30,5)*{\Gamma_{8b}},
 \endxy
\xy
(73,20);(65.8,11) **@{-}, (64.2,9);(62.6,7) **@{-}, (61,5);(57,0)
**@{-}, (67,20);(73,12) **@{-}, (73,12);(68.8,7) **@{-},
(67.3,5.2);(65.6,3.2) **@{-}, (64.3,1.6);(63,0) **@{-},
(63,20);(57,12) **@{-}, (67,0);(57,12) **@{-}, (73,0);(57,20)
**@{-},
 (58.5,17.7) *{\ulcorner}, %
 (68,18.8) *{\lrcorner}, 
 (59.1,9) *{\ulcorner}, %
 (62,18.4) *{\llcorner}, %
 (70.5,8.5) *{\urcorner}, 
 (71.8,18) *{\urcorner}, 
 (67,12.5) *{\llcorner}, 
 (71.7,1.6) *{\lrcorner},%
(60,18);(60,14) **@{-}, (60.1,18);(60.1,14) **@{-},
(59.9,18);(59.9,14) **@{-}, (60.2,18);(60.2,14) **@{-},
(59.8,18);(59.8,14) **@{-}, (68,16);(72,16) **@{-},
(68,16.1);(72,16.1) **@{-}, (68,15.9);(72,15.9) **@{-},
(68,16.2);(72,16.2) **@{-}, (68,15.8);(72,15.8) **@{-},
\endxy
\quad\quad
\xy (7,20);(14.2,11) **@{-}, (15.8,9);(17.4,7) **@{-},
(19,5);(23,0) **@{-}, (13,20);(7,12) **@{-}, (7,12);(11.2,7) **@{-},
(12.7,5.2);(14.4,3.2) **@{-}, (15.7,1.6);(17,0) **@{-},
(17,20);(23,12) **@{-}, (13,0);(23,12) **@{-}, (7,0);(23,20) **@{-},
(10,18);(10,14) **@{-}, (10.1,18);(10.1,14) **@{-},
(9.9,18);(9.9,14) **@{-}, (10.2,18);(10.2,14) **@{-},
(9.8,18);(9.8,14) **@{-}, (18,16);(22,16) **@{-},
(18,16.1);(22,16.1) **@{-}, (18,15.9);(22,15.9) **@{-},
(18,16.2);(22,16.2) **@{-}, (18,15.8);(22,15.8) **@{-},
 (8,18.8) *{\lrcorner}, %
 (18,18.8) *{\lrcorner}, 
 (9.1,9.5) *{\lrcorner}, %
 (11.8,18.1) *{\urcorner}, %
 (20.5,8.5) *{\urcorner}, 
 (21.8,18) *{\urcorner}, 
 (17,12.5) *{\llcorner}, 
 (21.7,1.3) *{\ulcorner},%
(27,8.3);(32,8.3) **@{-} ?>*\dir{>} ?<*\dir{<}, (37,12) *{},(30,5)*{\Gamma_{8c}},
 \endxy
\xy
(73,20);(65.8,11) **@{-}, (64.2,9);(62.6,7) **@{-}, (61,5);(57,0)
**@{-}, (67,20);(73,12) **@{-}, (73,12);(68.8,7) **@{-},
(67.3,5.2);(65.6,3.2) **@{-}, (64.3,1.6);(63,0) **@{-},
(63,20);(57,12) **@{-}, (67,0);(57,12) **@{-}, (73,0);(57,20)
**@{-},
(58,18.8) *{\lrcorner}, %
 (68,18.8) *{\lrcorner}, 
 (59.1,9.5) *{\lrcorner}, %
 (61.8,18.1) *{\urcorner}, %
 (70.5,8.5) *{\urcorner}, 
 (71.8,18) *{\urcorner}, 
 (67,12.5) *{\llcorner}, 
 (71.7,1.3) *{\ulcorner},%
(60,18);(60,14) **@{-}, (60.1,18);(60.1,14) **@{-},
(59.9,18);(59.9,14) **@{-}, (60.2,18);(60.2,14) **@{-},
(59.8,18);(59.8,14) **@{-}, (68,16);(72,16) **@{-},
(68,16.1);(72,16.1) **@{-}, (68,15.9);(72,15.9) **@{-},
(68,16.2);(72,16.2) **@{-}, (68,15.8);(72,15.8) **@{-},
\endxy}
\caption{Oriented Yoshikawa moves of type $4$-$8$}
\label{fig-om2}
\end{figure}

We first prove Theorem \ref{main-thm-1-2} and then Theorem \ref{main-thm-1-3}. The proof of Theorem \ref{main-thm-1-2} consists of showing that all oriented Yoshikawa moves in Fig.~\ref{fig-om2} can be realized by a finite sequence of plane isotopies, the moves $\Gamma_i (1 \leq i \leq 8)$ and $\Gamma_i' (i=1,4,6)$. This will proceed in several lemmas as the followings.

\begin{lemma}\label{main-thm-3-3-lem1}
The move $\Gamma_{4a}$ is realized by a finite sequence of the moves $\Gamma_{2a},$ $\Gamma_{2b},$ $\Gamma_{3b},$ $\Gamma_{3c},$ $\Gamma_{4},$ $\Gamma_{5}$ and plane isotopies.
The move $\Gamma'_{4a}$ is realized by a finite sequence of the moves $\Gamma_{2},$ $\Gamma_{2b},$ $\Gamma_{3e},$ $\Gamma_{3f},$ $\Gamma'_{4},$ $\Gamma_{5}$ and plane isotopies.
\end{lemma}

\begin{proof}
See Fig.~\ref{fig-pf-oym4-2}. 
\end{proof}

\begin{figure}[h]
\begin{center}
\resizebox{0.85\textwidth}{!}{%
  \includegraphics{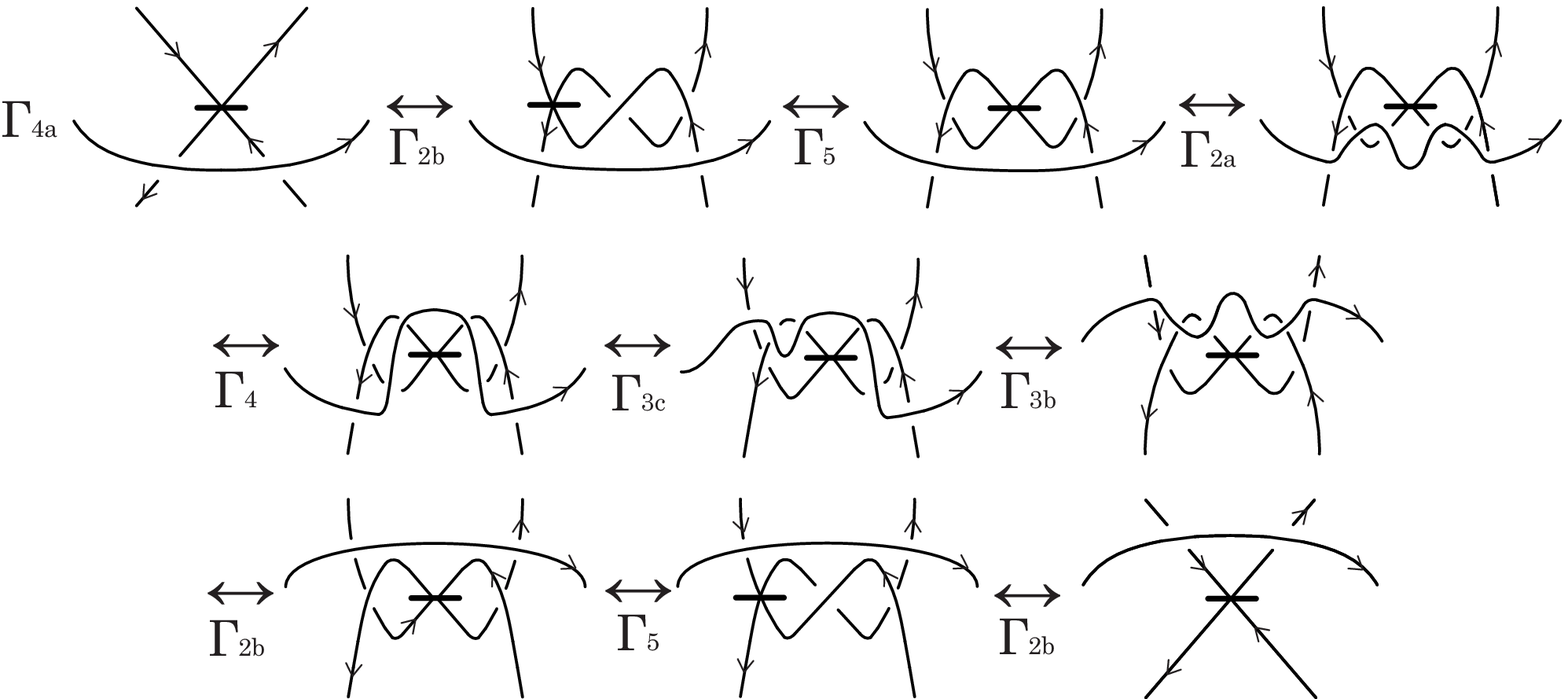} }
\vskip 0.3cm
\resizebox{0.85\textwidth}{!}{%
  \includegraphics{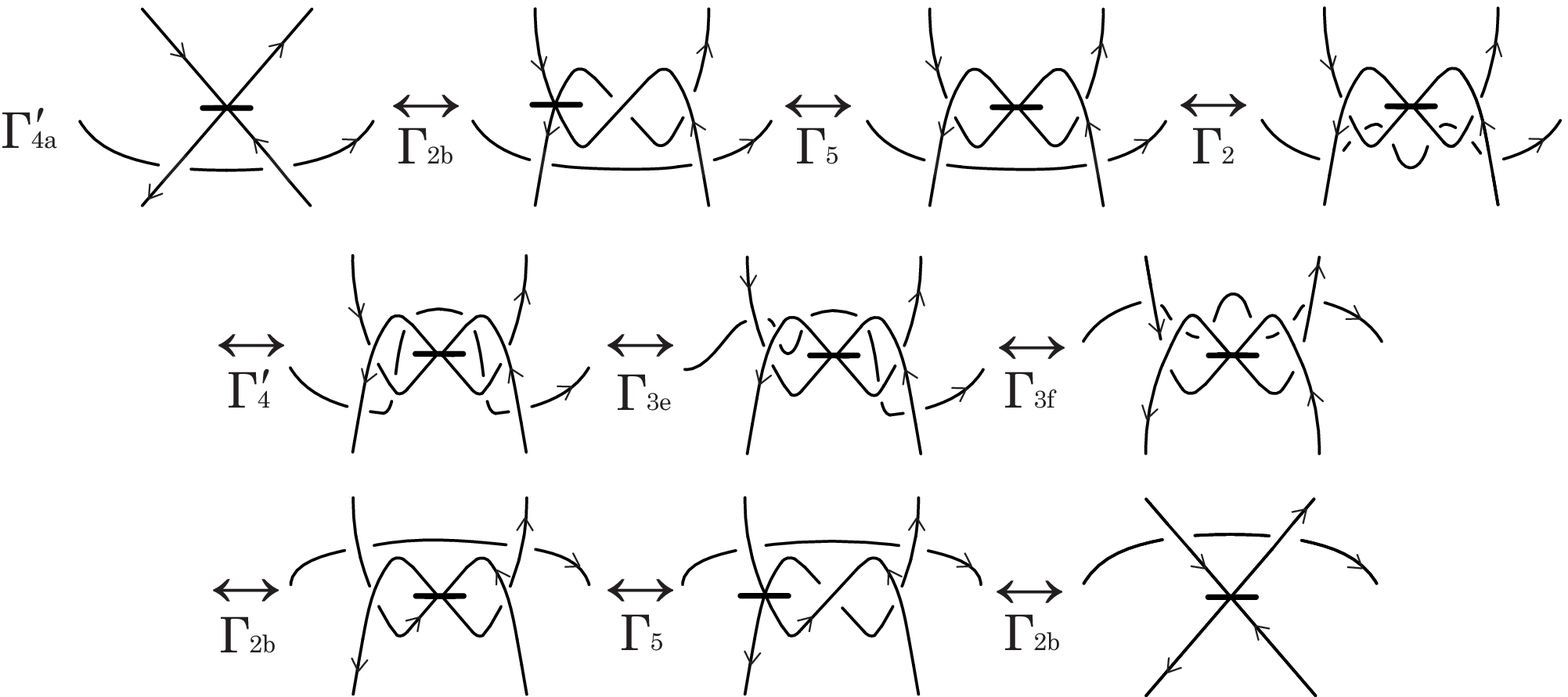} }
\caption{Oriented Yoshikawa moves $\Gamma_{4a}$ and $\Gamma'_{4a}$}\label{fig-pf-oym4-2}
\end{center}
\end{figure}

\begin{lemma}\label{main-thm-3-3-lem4}
Let $\Gamma_{5a}$ denote the oriented Yoshikawa move of type 5 shown in Fig.~\ref{fig-pf-oym5-1}. Then it is realized by a finite sequence of the moves $\Gamma'_{1},$ $\Gamma_{1b},$  $\Gamma_{2a},$ $\Gamma_{2b},$ $\Gamma_{3b},$ $\Gamma_{4},$ $\Gamma'_{4},$ $\Gamma_{5}$ and plane isotopies.
\begin{figure}[h]
\begin{center}
\resizebox{0.40\textwidth}{!}{%
\includegraphics{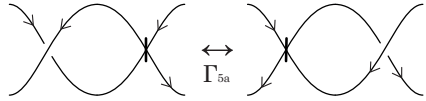} }
\caption{Oriented Yoshikawa move $\Gamma_{5a}$}
\label{fig-pf-oym5-1}
\end{center}
\end{figure}
\end{lemma}

\begin{proof}
See Fig.~\ref{fig-pf-oym5-2}.
\end{proof}

\begin{figure}[h]
\begin{center}
\resizebox{0.90\textwidth}{!}{%
  \includegraphics{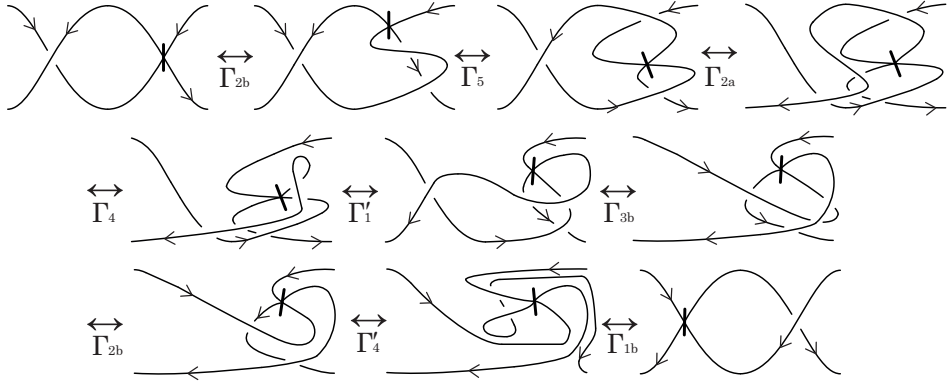} }
\caption{A realization of $\Gamma_{5a}$}
\label{fig-pf-oym5-2}
\end{center}
\end{figure}

\begin{lemma}\label{main-thm-3-3-lem7}
The move $\Gamma_{6a}$ is realized by a finite sequence of the moves $\Gamma'_{1},$ $\Gamma_{1b},$ $\Gamma_{2},$ $\Gamma_{2a},$ $\Gamma_{2b},$ $\Gamma_{3b},$ $\Gamma_{4},$ $\Gamma'_{4},$ $\Gamma_{5}$, $\Gamma_{6}$ and plane isotopies.
The move $\Gamma'_{6a}$ is realized by a finite sequence of the moves $\Gamma'_{1},$ $\Gamma_{5},$ $\Gamma'_{6}$. 
\end{lemma}

\begin{proof}
It follows from Fig.~\ref{fig-pf-oym6-2} that the moves $\Gamma_{6a}$ is realized by the moves $\Gamma'_{1},$ $\Gamma_{5a},$ $\Gamma_{6}$ and plane isotopies. By Lemma \ref{main-thm-3-3-lem4}, the move $\Gamma_{5a}$ is realized by a finite sequence of the moves $\Gamma'_{1},$ $\Gamma_{1b},$ $\Gamma_{2},$ $\Gamma_{2a},$ $\Gamma_{2b},$ $\Gamma_{3b},$ $\Gamma_{4},$ $\Gamma'_{4},$ $\Gamma_{5}$ and plane isotopies. This gives the assertion for $\Gamma_{6a}$. It is immediate from Fig.~\ref{fig-pf-oym6-2} that the move $\Gamma'_{6a}$ is realized by a finite sequence of the moves $\Gamma'_{1},$ $\Gamma_{5},$ $\Gamma'_{6}$. This completes the proof.
\end{proof}

\begin{figure}[h]
\begin{center}
\resizebox{0.75\textwidth}{!}{%
  \includegraphics{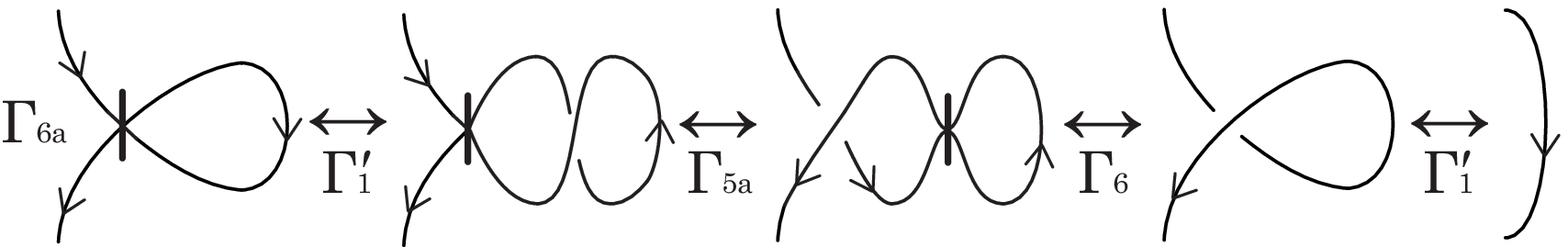} }
\vskip 0.3cm
\resizebox{0.75\textwidth}{!}{%
  \includegraphics{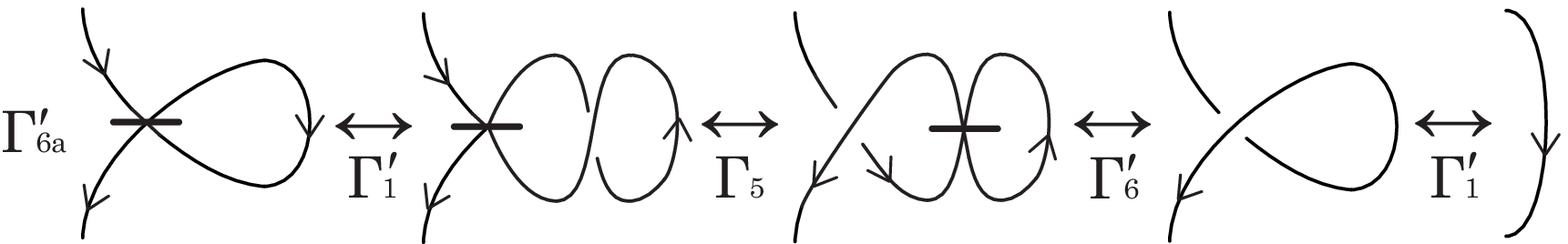} }
\caption{Oriented Yoshikawa move $\Gamma_{6a}$ and $\Gamma'_{6a}$}
\label{fig-pf-oym6-2}
\end{center}
\end{figure}

\begin{lemma}\label{main-thm-3-3-lem8}
The move $\Gamma_{7a}$ is realized by a finite sequence of the moves $\Gamma'_{1},$ $\Gamma_{1a},$ $\Gamma_{1b},$ $\Gamma_{2},$ $\Gamma_{2b},$ $\Gamma_{2c},$ $\Gamma_{3b},$ $\Gamma_{4},$ $\Gamma'_{4},$ $\Gamma_{4a},$ $\Gamma_{5},$ $\Gamma_{5a},$ $\Gamma_{7}$ and plane isotopies.
\end{lemma}

\begin{proof}
It follows from Fig.~\ref{fig-pf-oym7} that the move $\Gamma_{7a}$ is realized by a finite sequence of the moves $\Gamma_{1a},$ $\Gamma'_{1},$ $\Gamma_{2b},$ $\Gamma_{2c},$ $\Gamma_{4a},$ $\Gamma'_{4},$ $\Gamma_{5},$ $\Gamma_{5a},$ $\Gamma_{7}$ and plane isotopies. By Lemma \ref{main-thm-3-3-lem4}, the move $\Gamma_{5a}$ can be replaced with the moves $\Gamma'_{1},$ $\Gamma_{1b},$ $\Gamma_{2},$ $\Gamma_{2a},$ $\Gamma_{2b},$ $\Gamma_{3b},$ $\Gamma_{4},$ $\Gamma'_{4},$ $\Gamma_{5}$. This gives the assertion of the lemma.
\end{proof}

\begin{figure}
\begin{center}
\resizebox{0.95\textwidth}{!}{%
\includegraphics{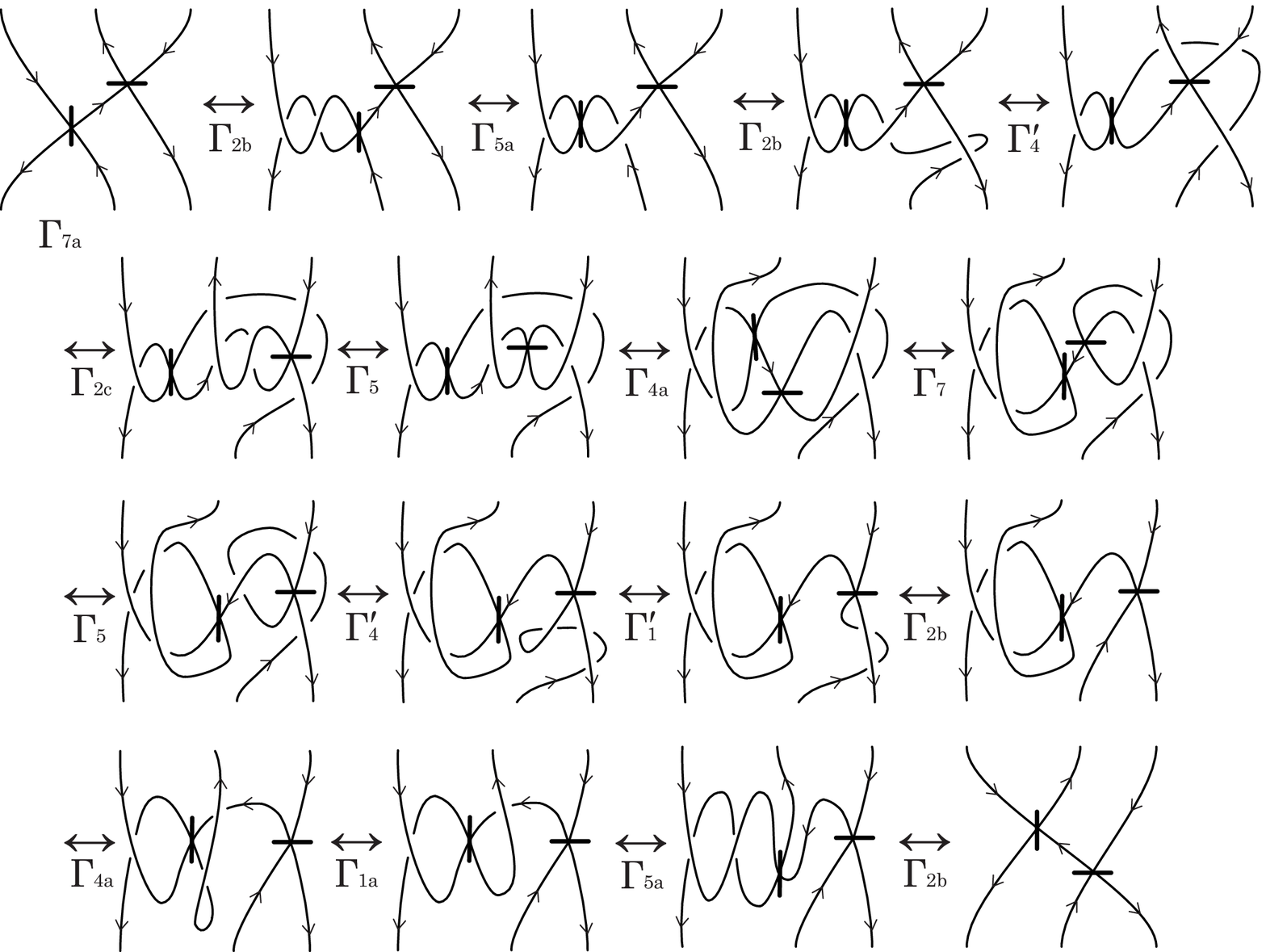} }
\caption{Oriented Yoshikawa move $\Gamma_{7a}$}
\label{fig-pf-oym7}
\end{center}
\begin{center}
\resizebox{0.80\textwidth}{!}{%
  \includegraphics{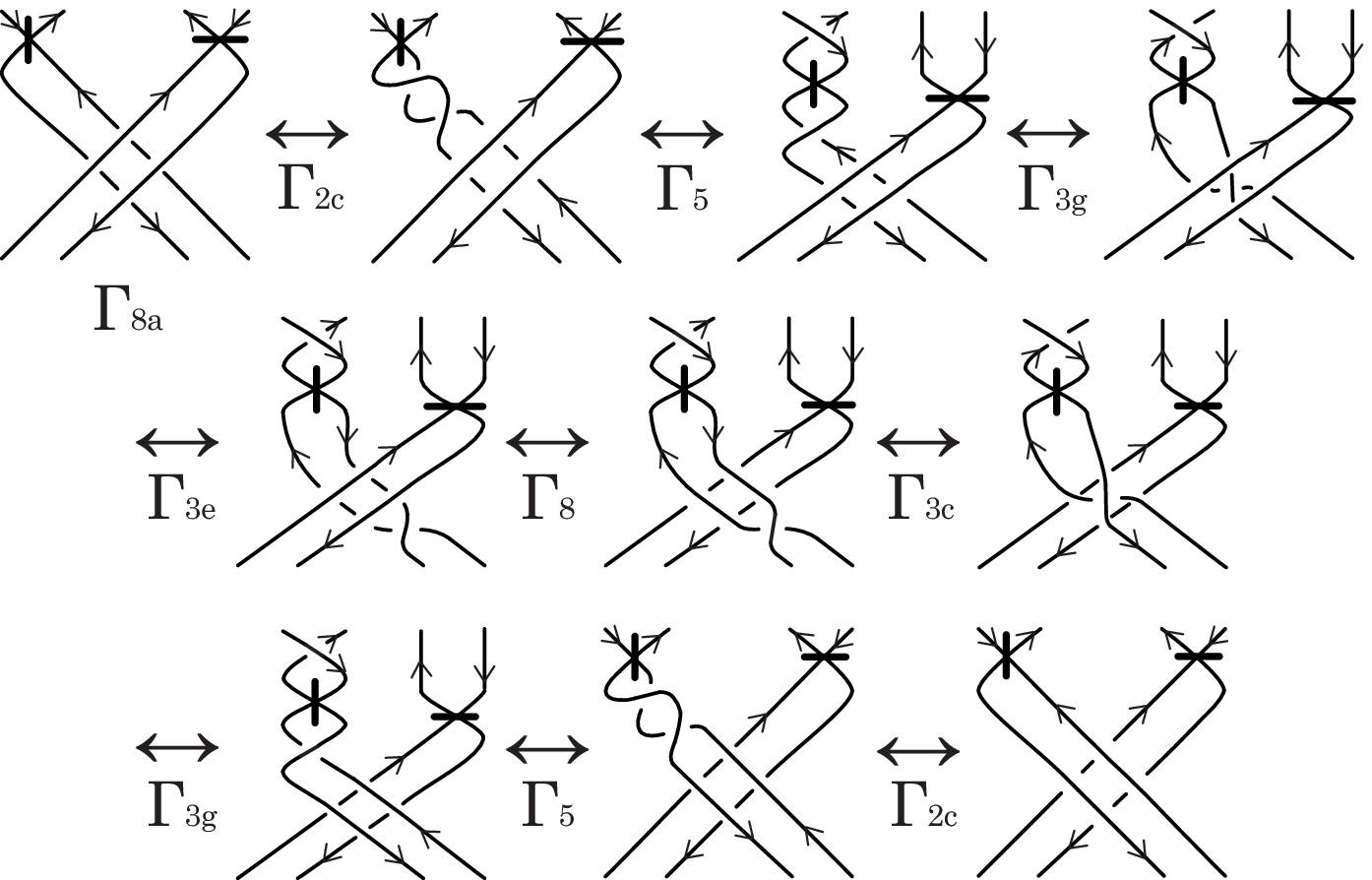} }
\caption{Oriented Yoshikawa move $\Gamma_{8a}$}
\label{fig-pf-oym8-1}
\end{center}
\end{figure}

\begin{lemma}\label{main-thm-3-3-lem9}
The move $\Gamma_{8a}$ is realized by a finite sequence of the moves
$\Gamma_{2c},$ $\Gamma_{3c},$ $\Gamma_{3g},$ $\Gamma_{3e},$ $\Gamma_{5},$ $\Gamma_{8}$ and plane isotopies.
The move $\Gamma_{8b}$ is realized by a finite sequence of the moves
$\Gamma_{2b},$ $\Gamma_{3c},$ $\Gamma_{3e},$ $\Gamma_{3g},$ $\Gamma'_{1},$ $\Gamma_{1b},$ $\Gamma_{2},$ $\Gamma_{2a},$ $\Gamma_{2b},$ $\Gamma_{3b},$ $\Gamma_{4},$ $\Gamma'_{4},$ $\Gamma_{5}$, $\Gamma_{8a}$ and plane isotopies.
The move $\Gamma_{8c}$ is realized by a finite sequence of the moves
$\Gamma_{2b},$ $\Gamma_{3c},$ $\Gamma_{3e},$ $\Gamma_{3g},$ $\Gamma'_{1},$ $\Gamma_{1b},$ $\Gamma_{2},$ $\Gamma_{2a},$ $\Gamma_{2b},$ $\Gamma_{3b},$ $\Gamma_{4},$ $\Gamma'_{4},$ $\Gamma_{5}$, $\Gamma_{8a}$ and plane isotopies.
\end{lemma}

\begin{proof}
It is immediate from Fig.~\ref{fig-pf-oym8-1} that the move $\Gamma_{8a}$ is realized by a finite sequence of the moves $\Gamma_{2c},$ $\Gamma_{3c},$ $\Gamma_{3g},$ $\Gamma_{3e},$ $\Gamma_{5},$ $\Gamma_{8}$ and plane isotopies. This gives the assertion for $\Gamma_{8a}$. Now it follows from Fig.~\ref{fig-pf-oym8-3} that the move $\Gamma_{8b}$ is realized by a finite sequence of the moves $\Gamma_{2b},$ $\Gamma_{3c},$ $\Gamma_{3e},$ $\Gamma_{3g},$ $\Gamma_{5a}$, $\Gamma_{8a}.$ By Lemma \ref{main-thm-3-3-lem4}, the move $\Gamma_{5a}$ can be replaced with the moves $\Gamma'_{1},$ $\Gamma_{1b},$ $\Gamma_{2},$ $\Gamma_{2a},$ $\Gamma_{2b},$ $\Gamma_{3b},$ $\Gamma_{4},$ $\Gamma'_{4},$ $\Gamma_{5}$. This yields the assertion for $\Gamma_{8b}$. Similarly, we obtain the assertion for the move $\Gamma_{8c}$ (see Fig.~\ref{fig-pf-oym8-3}). This completes the proof.
\end{proof}

\begin{figure}
\begin{center}
\resizebox{0.80\textwidth}{!}{%
  \includegraphics{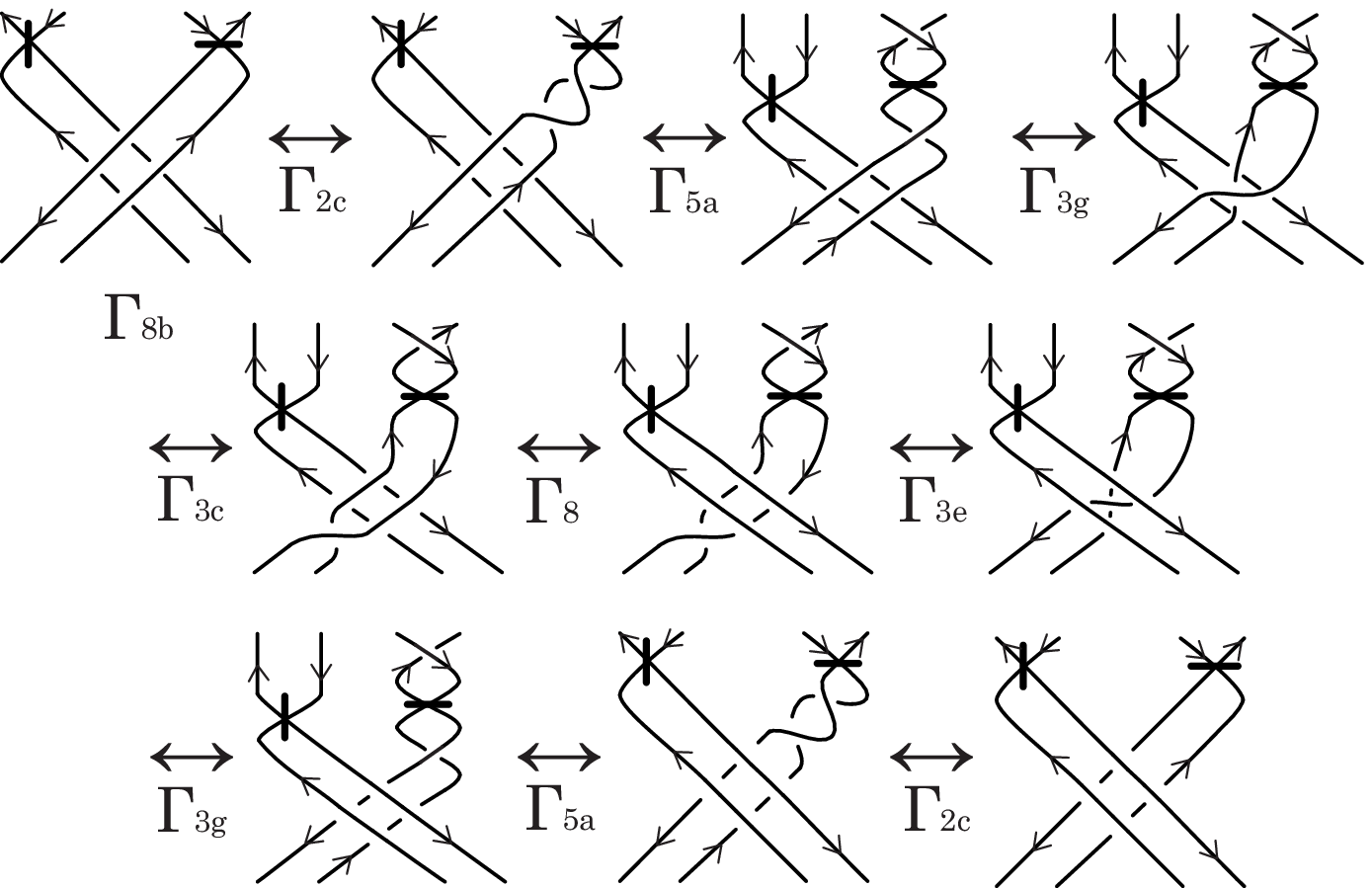} }
\caption{Oriented Yoshikawa move $\Gamma_{8b}$}
\label{fig-pf-oym8-2}
\end{center}
\begin{center}
\resizebox{0.80\textwidth}{!}{%
  \includegraphics{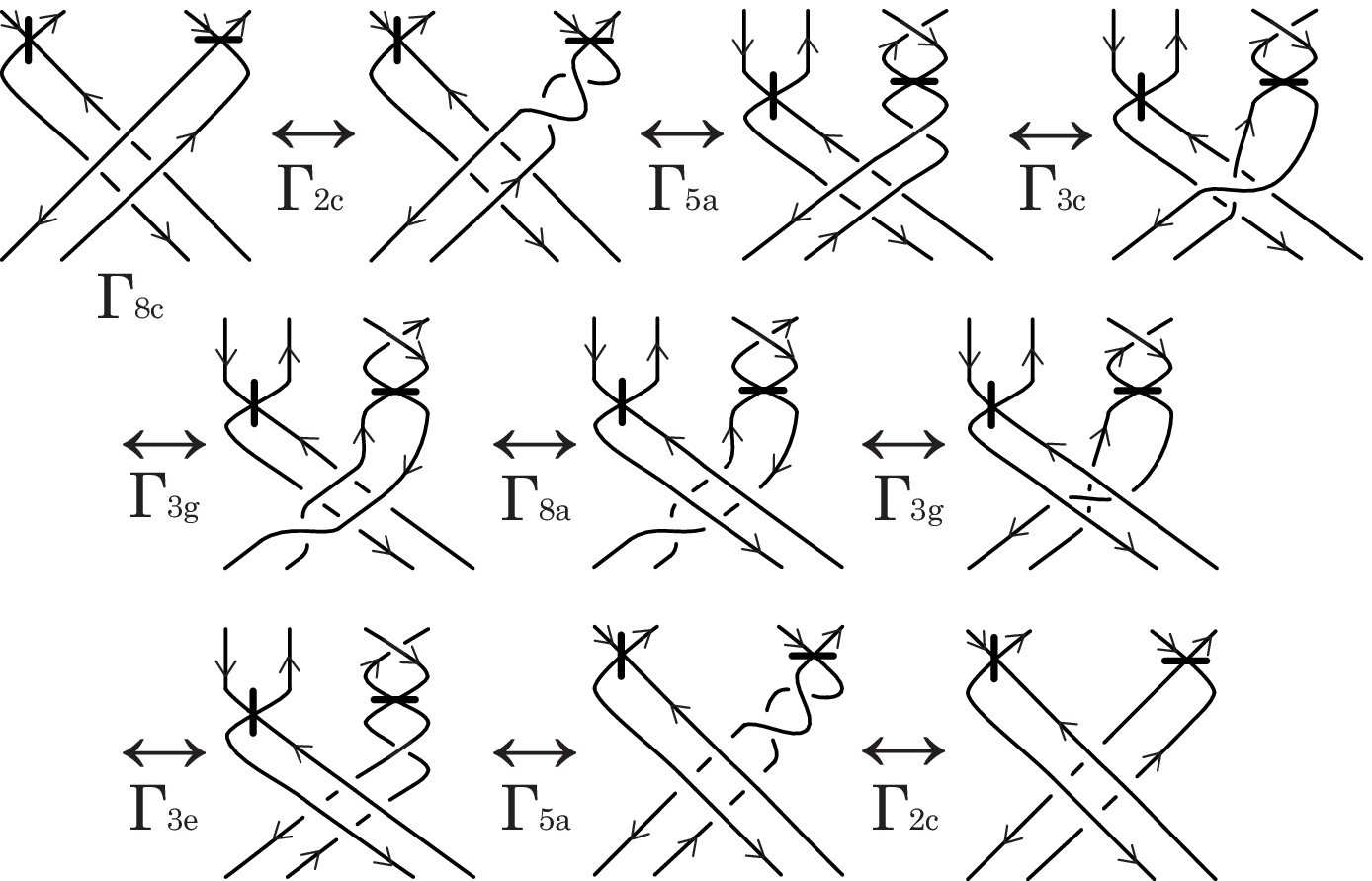} }
\caption{Oriented Yoshikawa move $\Gamma_{8b}$ and $\Gamma_{8c}$}
\label{fig-pf-oym8-3}
\end{center}
\end{figure}

Now let us complete the proof of Theorem \ref{main-thm-1-2} and Theorem \ref{main-thm-1-3}.

\bigskip

\indent{\bf Proof of Theorem \ref{main-thm-1-2}.} 
Let $M$ be a move of the oriented Yoshikawa moves. 
From Theorem \ref{main-thm-1-1}, we see that $M$ without orientation is generated by the only $10$ unoriented moves $\Omega_i, \Omega'_i \in \mathfrak S$. This gives that $M$ is generated by the 10 moves $\Omega_1, \Omega_2, \Omega_3, \Omega_4, \Omega'_4$, $\Omega_5$, $\Omega_6, \Omega'_6, \Omega_7, \Omega_8$ with orientations. In Fig.~\ref{fig-moves-type-II-o}, Fig.~\ref{fig-r123o} and Fig.~\ref{fig-om2}, these 10 moves with all possible orientations are found.
From Theorem \ref{thm-polyak}, we see that the classical Reidemeister moves $\Omega_1, \Omega_2, \Omega_3$ with orientation are generated by the moves $\Gamma_1, \Gamma'_1, \Gamma_2$ and $\Gamma_3$. For other oriented moves, it is seen from Lemmas \ref{main-thm-3-3-lem1}, \ref{main-thm-3-3-lem7}, \ref{main-thm-3-3-lem8}, \ref{main-thm-3-3-lem9} that the moves are realized by a finite sequence of oriented moves $\Gamma_i$ and $\Gamma'_i$ listed in $\mathfrak S_1$. This completes the proof of Theorem \ref{main-thm-1-2}.\hfill$\Box$

\bigskip

\indent{\bf Proof of Theorem \ref{main-thm-1-3}.} 
Let $M$ be a move of the oriented Yoshikawa moves. In the proof of Theorem \ref{main-thm-1-2}, we see that $M$ is generated by the 10 moves $\Omega_1, \Omega_2, \Omega_3,$ $\Omega_4,$ $\Omega'_4$, $\Omega_5$, $\Omega_6, \Omega'_6, \Omega_7, \Omega_8$ with orientations in Fig.~\ref{fig-moves-type-II-o}, Fig.~\ref{fig-r123o} and Fig.~\ref{fig-om2}. From Theorem \ref{thm-polyak-1}, we see that the five oriented Reidemeister moves $\Gamma_1, \Gamma'_1, \Gamma_{2b}, \Gamma_{2c}, \Gamma_{3a}$ generate all oriented Reidemeister moves. Thus the assertion follows from the previous theorem \ref{main-thm-1-2} at once. This completes the proof of Theorem \ref{main-thm-1-3}. \hfill$\Box$

\begin{remark}
There are various generating sets of classical oriented Reidemeister moves other than $\{\Gamma_1, \Gamma'_1, \Gamma_{2b}, \Gamma_{2c}, \Gamma_{3a}\}$ contianing the non-cyclic $\Omega_{3a}$ move $\Gamma_{3a}$ (see Theorem \ref{thm-polyak-1}). Replacing any generating set of oriented Reidemeister moves with the moves $\Gamma_1, \Gamma'_1, \Gamma_{2b}, \Gamma_{2c}, \Gamma_{3a}$ yields a generating set of oriented Yoshikawa moves. This fact and Lemmas \ref{main-thm-3-3-lem1}-\ref{main-thm-3-3-lem9} enable us to show that we can choose a generating set of oriented Yoshikawa moves other than ones given in Theorem \ref{main-thm-1-2} and Theorem \ref{main-thm-1-3}. For example, it is easily seen from Theorem \ref{thm-polyak-1} that $\{\Gamma_{1}, \Gamma_{1a}, \Gamma_{2b},  \Gamma_{2c}, \Gamma_{3a}, \Gamma_4, \Gamma'_4, \Gamma_5$, $\Gamma_6, \Gamma'_6, \Gamma_7, \Gamma_8\}$, $\{\Gamma_{1b}, \Gamma'_{1}, \Gamma_{2b},  \Gamma_{2c}, \Gamma_{3a}, \Gamma_4, \Gamma'_4, \Gamma_5$, $\Gamma_6, \Gamma'_6, \Gamma_7, \Gamma_8\}$ and $\{\Gamma_{1a}, \Gamma_{1b}, \Gamma_{2b}, \Gamma_{2c},$\\ $ \Gamma_{3a}, \Gamma_4, \Gamma'_4, \Gamma_5$, $\Gamma_6, \Gamma'_6, \Gamma_7, \Gamma_8\}$ are all generating sets of oriented Yoshikawa moves.
\end{remark}


\section{Independence of Yoshikawa moves}\label{sect-ind-ym}

For classical links in $\mathbb R^3$, it is well known that each Reidemeister move cannot be generated by the other two type moves, that is, a Reidemeister move is independent from the other types (cf. \cite[Appendix A]{Ma}). 

For Roseman moves on broken surface diagrams of surface-links in $\mathbb R^4,$ there also exist analogue results. Indeed, let $\mathcal L$ be a surface-link in $\mathbb R^4$ and let $\pi : \mathbb R^4 \to \mathbb R^3$ be the projection defined by $\pi(x_1,x_2,x_3,x_4) = (x_1,x_2,x_3).$ By a slight perturbation of $\mathcal L$ if necessary, we may assume that the restriction $\pi|_\mathcal L : \mathcal L \to \mathbb R^3$ is a generic map, i.e., the sigularities of $\pi|_\mathcal L$ consist of double point curves and isolated triple or branch points. The image $\pi(\mathcal L)$ in which the lower sheets are consistently removed along the double point curves is called a {\it broken surface diagram} of $\mathcal L$. For more details, see \cite{CS,Ka2}.
In \cite{Ro}, D. Roseman proved that two surface-links are equivalent if and only if their broken surface diagrams are transformed into each other by applying a finite number of seven types of local moves on broken surface diagrams which are called the {\it Roseman moves}. In \cite{Ya}, T. Yashiro proved that Roseman move of type 2 can be generated by the moves of type 1 and type 4. Recently, K. Kawamura showed that Roseman move of type 1 can be generated by the moves of type 2 and type 3, and for any $i (i=3,4,5,6,7)$,  Roseman move of type i cannot be generated by the other six types, that is, Roseman move of type $i$ is independent from the other six type moves \cite{Kar}.

For Yoshikawa moves on marked graph diagrams, it is natural to ask whether each Yoshikawa move is independent from the other moves.
In Theorem \ref{main-thm-1-1}, we have proved that $\mathfrak S=\{\Omega_1, \ldots, \Omega_8, \Omega'_4, \Omega'_6\}$ is a generating set of all unoriented Yoshikawa moves. From now on we shall discuss the independence of each Yoshikawa move in $\mathfrak S$ from the other moves in $\mathfrak S$. 

\begin{theorem}\label{main-thm-6-1}
The Yoshikawa move $\Omega_1$ is independent from the other moves in $\mathfrak S$.
\end{theorem} 

\begin{proof}
It is clear that all moves in $\mathfrak S$ except $\Omega_1$ keep the parity of the number of crossings and the move $\Omega_1$ changes the number of crossing by one and so does the parity. This gives that $\Omega_1$ cannot be realized by the moves in $\mathfrak S-\{\Omega_1\}$.
\end{proof}

\begin{theorem}\label{main-thm-6-2}
The Yoshikawa move $\Omega_2$ is independent from the other moves in $\mathfrak S$.
\end{theorem} 

\begin{proof}
Let $D=D_1 \cup D_2$ be a two component marked graph diagram in $\mathbb R^2$. By $|D|=|D_1| \cup |D_2|$ we denote the $4$-valent graph in $\mathbb R^2$ obtained from $D$ by removing markers and by assuming crossings to be verticess of valency $4$. Define
$$s(D)=
\begin{cases}
0 & \text{if $|D_1| \cap |D_2| =\emptyset$},\\ 
1 & \text{otherwise}.
\end{cases}
$$
Then it is easily seen that $s(D)$ is invariant under all moves in $\mathfrak S$ except $\Omega_2$. Now let $D=\bigcirc~\bigcirc$ and 
$D'=\xy (0,0) *\xycircle(4,4){-}, 
(1,2.5);(1,-2.5) **\crv{(-1,0)}, 
(4,3.7);(6,3.5) **\crv{(4.5,3.9)}, 
(4,-3.7);(6,-3.5) **\crv{(4.5,-3.9)}, 
(6,3.5);(6,-3.5) **\crv{(10,0)},
\endxy$. Then $D$ and $D'$ are two diagrams of the trivial $S^2$-link with two components, and $s(D)=0$ and $s(D')=1$. This gives that $s(D)$ is not an invariant of the move $\Omega_2$. Hence $\Omega_2$ cannot be realized by the moves in $\mathfrak S-\{\Omega_2\}$. 
\end{proof}

Let $D=D_1 \cup D_2 \cup D_3$ be a marked graph diagram in $\mathbb R^2$ with three components. Let $T(D)=\{p(D_1), p(D_2), p(D_3)\}$ be the non-ordered triple of numbers with $p(D_i) \in \{0,1\}$ defined as follows. For each $i=1,2,3$, the component $D_i$ tiles $\mathbb R^2$ into regions which admit a checkboard coloring. We color the unbounded region to be white. Let $B(D_i)$ denote the set of all black regions and let $n(D_i)$ be the number of crossings of two components different from $D_i$ and lie in $B(D_i)$. We define $p(D_i)=n(D_i) \pmod 2$. Then we have the following:

\begin{lemma}\label{main-lem-6-3}
 The non-ordered triple $T(D)=\{p(D_1), p(D_2), p(D_3)\}$ is invariant under all moves in $\mathfrak S$ except the move $\Omega_3$. 
\end{lemma}

\begin{proof}
For all moves in $\mathfrak S$ except the move $\Omega_3$, it is easily seen that the moves involves only one or two components. This straightforwardly gives that the moves contributing the value $T(D)$ are performed in a region with the same color. Since the all moves contains even number of crossings, the parity of $n(D_i)$ does not changed for each $i=1,2,3$. This completes the proof.
\end{proof}

\begin{theorem}\label{main-thm-6-3}
The Yoshikawa moves $\Omega_3$ is independent from the other moves in $\mathfrak S$.
\end{theorem}  

\begin{proof}
Let $D=D_1 \cup D_2 \cup D_3$ and $D'=D'_1 \cup D'_2 \cup D'_3$ be the diagrams of the trivial $S^2$-link with three components as shown in Fig.~\ref{ind-ym3}.
\begin{figure}
\begin{center}
\resizebox{0.6\textwidth}{!}{%
  \includegraphics{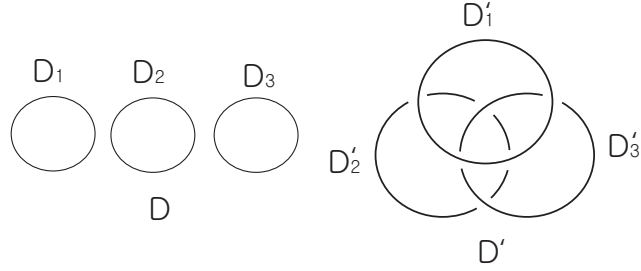}}
\caption{Diagrams of the trivial $S^2$-link with three components}\label{ind-ym3}
\end{center}
\end{figure}
Then it is easily seen that $p(D_1)=p(D_2)=p(D_3)=1$ and so $T(D)=\{1\}$. But $p(D'_1)=p(D'_2)=p(D'_3)=0$ and so $T(D')=\{0\}$. By Lemma \ref{main-lem-6-3}, $D$ cannot be transformed into $D'$ by using a sequence of the moves in $\mathfrak S-\{\Omega_3\}$. This concludes that the move $\Omega_3$ is independent from the moves in $\mathfrak S-\{\Omega_3\}$.
\end{proof}

\begin{theorem}\cite[Corollary 6.3]{JKaL}\label{main-thm-6-4}
The Yoshikawa move $\Omega_4$ and $\Omega'_4$ are independent from the moves in $\mathfrak S-\{\Omega_4, \Omega'_4\}$. 
\end{theorem}

\begin{theorem}\label{main-thm-6-6}
The Yoshikawa move $\Omega_6$ is independent from the other moves in $\mathfrak S$. The Yoshikawa move $\Omega'_6$ is independent from the other moves in $\mathfrak S$.
\end{theorem} 

\begin{proof}
Let $D$ be a marked graph diagram and let $\mu_+(D)$ and $\mu_-(D)$ denote the numbers of components of the positive resolution $L_+(D)$ and the negative resolution $L_-(D)$ of $D$, respectively. Then it is easily seen that $\mu_+(D)$ is invariant under all moves in $\mathfrak S$ except $\Omega_6$ and $\mu_-(D)$ is invariant under all moves in $\mathfrak S$ except $\Omega'_6$. But, the move $\Omega_6$ changes the number of components of $L_+(D)$ by one, and $\Omega'_6$ changes the number of components of $L_-(D)$ by one. This implies the assertions.
\end{proof}

\begin{theorem}\cite[Theorem 5.4]{JKL}\label{main-thm-6-7}
The Yoshikawa move $\Omega_7$ is independent from the other moves in $\mathfrak S$.
\end{theorem}

Let $D$ be a marked graph diagram and let $\tilde D$ be a classical link diagram obtained form $D$ by replacing each marked vertex $\xy (-5,5);(5,-5) **@{-}, 
(5,5);(-5,-5) **@{-},  
(3,-0.2);(-3,-0.2) **@{-},
(3,0);(-3,0) **@{-}, 
(3,0.2);(-3,0.2) **@{-}, 
\endxy$ to the crossing 
$\xy (-5,5);(5,-5) **@{-}, 
(5,5);(2,1.9) **@{-}, (-5,-5);(-2,-2.1) **@{-},   
\endxy.$ Let $\sharp D$ denote the number of the components of $\tilde D$. Then it is obvious that $\sharp D$ is an invariant of $D$ under all Yoshikawa moves in $\mathfrak S$ except the move $\Omega_7$. It is not difficult to find two marked graph diagrams $D$ and $D'$ that are differ by a single $\Omega_7$ and satisfy $\sharp D \not= \sharp D'$. This gives an alternative proof of Theorem \ref{main-thm-6-7}.

\begin{corollary}
Let $\Gamma$ be an oriented Yoshikawa move in $\mathfrak S_1-\{\Gamma_5, \Gamma_8\}$ ($\mathfrak S_2-\{\Gamma_5, \Gamma_8\}$, resp.). Then $\Gamma$ is independent from the other oriented moves in $\mathfrak S_1$ ($\mathfrak S_2$. resp.). 
\end{corollary}

\begin{proof}
Suppose that $\Gamma$ is realized by a finite sequence $M_1, M_2, \ldots, M_k$ with $k \geq 1$ and $M_i \in \mathfrak S_1-\{\Gamma\}$ ($\mathfrak S_2-\{\Gamma\}$. resp.) for all $1\leq i \leq k$. Let $\overline{\Gamma}$ and $\overline{M}_i$ be the moves $\Gamma$ and $M_i$ forgetting the orientations, respectively. Then we see that $\overline{\Gamma}$ is realized by the finite sequence $\overline{M}_1, \overline{M}_2, \ldots, \overline{M}_k$. Since $\overline{\Gamma} \in \mathfrak S-\{\Omega_5, \Omega_8\}$ and $\overline{M}_i \in \mathfrak S-\{\overline{\Gamma}\}$, we have a contradiction. This completes the proof.
\end{proof}

We now close this section with the following four questions. If the answers of these questions are all affirmative, then the generating sets $\mathfrak S$, $\mathfrak S_1$ and $\mathfrak S_2$ are all minimal.

\bigskip

\noindent{\bf Question 1.}
Is the Yoshikawa move $\Omega_4$ independent from the other moves in $\mathfrak S$? 

\bigskip

\noindent{\bf Question 2.}
Is the Yoshikawa move $\Omega'_4$ independent from the other moves in $\mathfrak S$?

\bigskip

\noindent{\bf Question 3.}
Is the Yoshikawa moves $\Omega_5$ independent from the other moves in $\mathfrak S$?

\bigskip

\noindent{\bf Question 4.}
Is the Yoshikawa move $\Omega_8$ independent from the other moves in $\mathfrak S$?


\section*{Acknowledgments}
The authors would like to express their sincere gratitude to the referee for many valuable comments. 
This work was supported by Basic Science Research
Program through the National Research Foundation of Korea(NRF) funded
by the Ministry of Education, Science and Technology (2013R1A1A2012446).

\end{document}